\documentclass{amsart}
\usepackage{amsmath,amssymb}
\title{Quasi-Frobenius algebras in finite tensor categories}
\author{Kenichi Shimizu}
\address[K. Shimizu]{Department of Mathematical Sciences,
  Shibaura Institute of Technology \\
  307 Fukasaku, Minuma-ku, Saitama-shi, Saitama 337-8570, Japan.}
\email{kshimizu@shibaura-it.ac.jp}
\usepackage{mathtools}
\usepackage{color}
\usepackage{bbm,mathrsfs,stmaryrd}
\usepackage[setpagesize=false]{hyperref}
\usepackage{tikz}
\usepackage{tikz-cd}
\usetikzlibrary{calc}
\usepackage{amsthm}
\numberwithin{equation}{section}
\theoremstyle{plain}
\newtheorem{C}{}[section] 
\newtheorem{lemma}[C]{Lemma}

\newtheorem{theorem}[C]{Theorem}
\newtheorem{proposition}[C]{Proposition}
\newtheorem{corollary}[C]{Corollary}
\newtheorem{conjecture}[C]{Conjecture}
\theoremstyle{definition}
\newtheorem{definition}[C]{Definition}
\theoremstyle{remark}
\newtheorem{remark}[C]{Remark}
\newtheorem{example}[C]{Example}
\newcommand{\id}{\mathrm{id}}
\newcommand{\op}{\mathrm{op}}

\newcommand{\radj}{\mathtt{r}}
\newcommand{\rradj}{\mathtt{rr}}

\newcommand{\bfk}{\Bbbk}
\newcommand{\Hom}{\mathrm{Hom}}
\newcommand{\End}{\mathrm{End}}
\newcommand{\Ker}{\mathrm{Ker}}

\newcommand{\Nat}{\mathrm{Nat}}

\newcommand{\Rex}{\mathrm{Rex}}
\newcommand{\Lex}{\mathrm{Lex}}
\newcommand{\unitobj}{\mathbbm{1}}
\newcommand{\rev}{\mathrm{rev}}
\newcommand{\eval}{\mathrm{ev}}
\newcommand{\coev}{\mathrm{coev}}
\newcommand{\Vect}{\mathbf{Vec}}
\newcommand{\Nak}{\mathbb{N}}
\newcommand{\Ser}{\mathbb{S}}

\newcommand{\iHom}{\underline{\Hom}}
\newcommand{\iEnd}{\underline{\End}}
\newcommand{\icoev}{\underline{\coev}}
\newcommand{\ieval}{\underline{\eval}}
\newcommand{\icomp}{\underline{\smash{\mathrm{comp}}}}
\newcommand{\itrace}{\underline{\mathrm{tr}}}
\newcommand{\relprj}[2]{\mathscr{P}_{#1}(#2)}
\newcommand{\relinj}[2]{\mathscr{I}_{#1}(#2)}
\newcommand{\catactl}{\mathbin{\triangleright}}
\newcommand{\catactr}{\mathbin{\triangleleft}}
\newcommand{\T}{\mathbb{T}}
\newcommand{\act}{\rho}

\newcommand{\lmod}[1]{\text{{$#1$}-mod}}
\newcommand{\rmod}[1]{\text{mod-{$#1$}}}
\newcommand{\bimod}[2]{\text{{$#1$}-mod-{$#2$}}}
\begin{document}

\begin{abstract}
  We introduce the notion of a quasi-Frobenius algebra in a finite tensor category $\mathcal{C}$ and give equivalent conditions for an algebra in $\mathcal{C}$ to be quasi-Frobenius. A quasi-Frobenius algebra in $\mathcal{C}$ is not necessarily Frobenius, however, we show that an algebra $A$ in $\mathcal{C}$ is quasi-Frobenius if and only if $A$ is Morita equivalent to a Frobenius algebra in $\mathcal{C}$. We also show that the class of symmetric Frobenius algebras in $\mathcal{C}$ is closed under the Morita equivalence provided that $\mathcal{C}$ is pivotal so that the symmetricity makes sense.
\end{abstract}

\maketitle


\section{Introduction}
\label{sec:introduction}

Frobenius algebras are an important class of algebras studied and applied in various areas of mathematics and mathematical physics.
Since a Frobenius algebra can be defined as a vector space coming with some linear maps subject to equations written by compositions and tensor products of linear maps, one can define a Frobenius algebra in any monoidal category; see \cite{MR2095680,MR2500035} for the precise definition and equivalent conditions.

An ordinary Frobenius algebra over a field $\bfk$ is nothing but a Frobenius algebra in the category $\Vect$ of finite-dimensional vector spaces over $\bfk$.
The category $\Vect$ is the simplest example of fusion categories \cite{MR2183279} or, more generally, finite tensor categories \cite{MR2119143,MR3242743}.
Frobenius algebras in finite tensor categories have been studied actively.
For example, M\"uger \cite{MR1966524,MR1966525} used special Frobenius algebras in fusion categories to capture some algebraic aspects of subfactor theory.
Commutative special Frobenius algebras appear in constructions of modular tensor categories by means of local modules \cite{MR1936496,2022arXiv220208644L}.
It was shown that a full rational conformal field theory (CFT) is given by a special symmetric Frobenius algebra in a modular tensor category \cite{MR1940282}; see also \cite{MR2342830,MR2342831} for expositions and related results.
Recently, non-special symmetric Frobenius algebras are also of interest from the viewpoint of CFTs and topological field theories. With regard to this, some constructions of Frobenius algebras in finite tensor categories are established in \cite{MR4586249,MR4681293,2019arXiv190400376S,MR4683832}.

Given an algebra $A$ in a finite tensor category $\mathcal{C}$, we denote by $\mathcal{C}_A$ the category of right $A$-modules in $\mathcal{C}$. There is a natural functor $\mathcal{C} \times \mathcal{C}_A \to \mathcal{C}_A$, which makes $\mathcal{C}_A$ a left $\mathcal{C}$-module category \cite{MR3242743}.
Two algebras $A$ and $B$ are said to be {\em Morita equivalent} if $\mathcal{C}_A$ and $\mathcal{C}_B$ are equivalent as left $\mathcal{C}$-module categories.
The Morita class of an algebra is sometimes essential rather than the algebra itself.
For instance, in the framework of \cite{MR1940282}, the full CFT arising from a special Frobenius algebra $A$ actually depends on the Morita class of $A$.

As is already in the case where $\mathcal{C} = \Vect$, the class of Frobenius algebras in a finite tensor category $\mathcal{C}$ is not closed under Morita equivalence.
Thus it is natural to ask what is the class of algebras in $\mathcal{C}$ that are Morita equivalent to Frobenius algebras in $\mathcal{C}$.
In this paper, we introduce the notion of {\em quasi-Frobenius algebra} in $\mathcal{C}$ by mimicking the definition of quasi-Frobenius rings. Our answer to the above question is as follows: An algebra in $\mathcal{C}$ is Morita equivalent to a Frobenius algebra if and only if it is quasi-Frobenius.

To explain our results in a bit more detail, we recall that the internal Hom functor $\iHom$ is defined for a finite left $\mathcal{C}$-module category $\mathcal{M}$ by
\begin{equation*}
  \Hom_{\mathcal{C}}(X, \iHom(M, N)) \cong \Hom_{\mathcal{M}}(X \catactl M, N)
  \quad (X \in \mathcal{C}, M, N \in \mathcal{M}).
\end{equation*}
An object $M \in \mathcal{M}$ is said to be $\mathcal{C}$-projective if the functor $\iHom(M, -)$ is exact \cite{MR3934626}. Similarly, $M$ is said to be $\mathcal{C}$-injective if $\iHom(-,M)$ is exact.
We recall that a quasi-Frobenius ring is a ring $R$ that is Noetherian and injective as a right $R$-module \cite[Chapter 6]{MR1653294}.
Due to the finiteness of $\mathcal{C}$, every algebra in $\mathcal{C}$ is Noetherian. We therefore propose to define:

\begin{definition}[$=$ Definition \ref{def:QF-algebras}]
  A {\em quasi-Frobenius algebra} in $\mathcal{C}$ is an algebra $A$ in $\mathcal{C}$ such that the right $A$-module $A$ is $\mathcal{C}$-injective.
\end{definition}

In this paper, we establish fundamental results on quasi-Frobenius algebras in $\mathcal{C}$.
One of main tools in this paper is a generalization of the relative Serre functor for an exact module category \cite{MR3435098,MR4042867,2019arXiv190400376S}.
For a left $\mathcal{C}$-module category which is not necessarily exact, we define the {\em relative Serre functor} of $\mathcal{M}$ to be the right exact endofunctor $\Ser$ on $\mathcal{M}$ such that there is a natural isomorphism
\begin{equation}
  \label{eq:intro-rel-Serre-def}
  \iHom(P, M)^* \cong \iHom(M, \Ser(P))
  \quad (P, M \in \mathcal{M}; \text{$P$ is $\mathcal{C}$-projective}),
\end{equation}
where $(-)^*$ is the left duality functor of $\mathcal{C}$.
The relative Serre functor is shown to exist and to be unique up to isomorphism.
Now let $A$ be an algebra in $\mathcal{C}$, and let $\Ser_A$ be the relative Serre functor of $\mathcal{C}_A$.
We note that $A^*$ has a natural structure of a right $A$-module in $\mathcal{C}$. The main result of this paper is as follows:

\begin{theorem}[$=$ Theorem~\ref{thm:QF-algebras}]
  The following conditions are equivalent:
  \begin{enumerate}
  \item The right $A$-module $A$ is $\mathcal{C}$-injective, that is, $A$ is quasi-Frobenius.
  \item The right $A$-module $A^*$ is $\mathcal{C}$-projective.
  \item Every $\mathcal{C}$-projective object of $\mathcal{C}_A$ is $\mathcal{C}$-injective.
  \item Every $\mathcal{C}$-injective object of $\mathcal{C}_A$ is $\mathcal{C}$-projective.
  \item The functor $\Ser_A$ is an equivalence.
  \item The functor $\Ser_A$ is exact.
  \item A right adjoint of $\Ser_A$ is exact.
  \item The Nakayama functor \cite{MR4042867} of $\mathcal{C}_A$ is an equivalence.
  \item The algebra $A$ is Morita equivalent to a Frobenius algebra in $\mathcal{C}$.
  \end{enumerate}
\end{theorem}

It is obvious that a Frobenius algebra in $\mathcal{C}$ is quasi-Frobenius.
A Hopf algebra in a braided finite tensor category is shown to be quasi-Frobenius by the integral theory for Hopf algebras.
In this paper, we also provide some results on $\mathcal{C}$-projective objects, $\mathcal{C}$-injective objects and the relative Serre functor.
As an application, we show that the class of symmetric Frobenius algebras in a pivotal finite tensor category is closed under the Morita equivalence.

\subsection*{Organization of this paper}

This paper is organized as follows:
In Section~\ref{sec:preliminaries}, we recall basic notions related to monoidal categories and module categories. We also provide some technical lemmas for natural isomorphisms coming with the internal Hom functor.

In Section~\ref{sec:ftc-and-modules}, we review mainly from \cite{2014arXiv1406.4204D,MR3242743,MR4042867} fundamentals of finite abelian categories, finite tensor categories, and finite module categories required in this paper, such as Morita theoretic methods, the Eilenberg-Watts equivalence, and the Nakayama functor.

Let $\mathcal{C}$ be a finite tensor category, and let $\mathcal{M}$ be a finite left $\mathcal{C}$-module category with action $\catactl : \mathcal{C} \times \mathcal{M} \to \mathcal{M}$.
All the results in Section~\ref{sec:rel-Serre} are generalizations of some results given in \cite{MR3435098,MR4042867,2019arXiv190400376S}, where $\mathcal{M}$ is assumed to be an exact module category.
In this section, without assuming the exactness of $\mathcal{M}$, we define the relative Serre functor $\Ser$ of $\mathcal{M}$ to be a right exact linear endofunctor on $\mathcal{M}$ such that there is a natural isomorphism \eqref{eq:intro-rel-Serre-def}. 
There is a natural isomorphism
\begin{equation*}
  X^{**} \catactl \Ser(M) \to \Ser(X \catactl M)
  \quad (X \in \mathcal{C}, M \in \mathcal{M})
\end{equation*}
making $\Ser$ a `twisted' left $\mathcal{C}$-module functor.
We show that the relative Serre functor of $\mathcal{M}$ is unique up to isomorphism of twisted left $\mathcal{C}$-module functors. Moreover, there is a natural isomorphisms $\Nak_{\mathcal{M}}(M) \cong \Nak_{\mathcal{C}}(\unitobj) \catactl \Ser(M)$ ($M \in \mathcal{M}$) of twisted left $\mathcal{C}$-module functors, where $\Nak_{\mathcal{A}}$ for a finite abelian category $\mathcal{A}$ is the Nakayama functor \cite{MR4042867} of $\mathcal{A}$.
The case where $\mathcal{M} = \mathcal{C}_A$ for some algebra $A$ in $\mathcal{C}$ is especially important. We define endofunctors $\Ser_A$ and $\overline{\Ser}_A$ on $\mathcal{C}_A$ by
\begin{equation}
  \label{eq:intro-rel-Serre-of-CA}
  \Ser_A(M) = \iHom_A(M, A)^*
  \quad \text{and} \quad
  \overline{\Ser}_A(M) = {}^{**}\iHom_A(A^*, M)
  \quad (M \in \mathcal{C}_A),
\end{equation}
respectively, where $\iHom_A$ is the internal Hom functor of $\mathcal{C}_A$. The functors $\Ser_A$ and $\overline{\Ser}_A$ are the relative Serre functor of $\mathcal{C}_A$ and its right adjoint, respectively.

In Section~\ref{sec:qf-algebras}, we introduce quasi-Frobenius algebras in $\mathcal{C}$ and give several equivalent conditions for an algebra to be quasi-Frobenius. We first provide a bunch of lemmas on $\mathcal{C}$-projective and $\mathcal{C}$-injective objects with reference to well-known results on projective and injective objects. By \eqref{eq:intro-rel-Serre-of-CA}, we show that a finite left $\mathcal{C}$-module category $\mathcal{M}$ is equivalent to $\mathcal{C}_A$ for some Frobenius algebra $A$ in $\mathcal{C}$ if and only if $\mathcal{M}$ has a $\mathcal{C}$-projective $\mathcal{C}$-generator $G$ such that $\Ser(G) \cong G$.
Based on these results, we give the characterization of quasi-Frobenius algebras described in the above. After giving some useful corollaries, we show that a Hopf algebra in a braided finite tensor category is quasi-Frobenius. We also raise a conjecture that a simple algebra in $\mathcal{C}$ is quasi-Frobenius, which is essentially equivalent to a conjecture given by Etingof and Ostrik in \cite{MR4237968}.

In Section~\ref{sec:symm-frobenius}, we assume that $\mathcal{C}$ is a pivotal finite tensor category and discuss symmetric Frobenius algebras in $\mathcal{C}$. In a similar way as \cite{MR3435098,2019arXiv190400376S}, we define a {\em pivotal structure} of a finite left $\mathcal{C}$-module category $\mathcal{M}$.
We then show that $\mathcal{M}$ is equivalent to $\mathcal{C}_A$ for some symmetric Frobenius algebra if and only if $\mathcal{M}$ admits a pivotal structure.
Hence we have a consequence that the class of symmetric Frobenius algebras in $\mathcal{C}$ is closed under Morita equivalence.

\subsection*{Acknowledgment}

The author thanks C. Schweigert, T. Shibata and H. Yadav for comments.
The author is supported by JSPS KAKENHI Grant Number JP20K03520.

\section{Preliminaries}
\label{sec:preliminaries}

\subsection{Monoidal categories and their modules}

Given a category $\mathcal{C}$, we write $X \in \mathcal{C}$ to mean that $X$ is an object of $\mathcal{C}$. An object $X \in \mathcal{C}$ is often denoted by $X^{\op}$ when it is viewed as an object of the opposite category $\mathcal{C}^{\op}$. Similar notation will be used for morphisms and functors.

Our main reference on monoidal categories is \cite{MR3242743}.
In view of Mac Lane's coherence theorem, all monoidal categories are assumed to be strict.
Let $\mathcal{C}$ be a monoidal category with monoidal product $\otimes$ and unit object $\unitobj$.
We note that $\mathcal{C}^{\op}$ has a natural structure of a monoidal category.
We denote by $\mathcal{C}^{\rev}$ the monoidal category whose underlying category is $\mathcal{C}$ but the order of the monoidal product is reversed.

According to \cite[Section 2.10]{MR3242743}, a {\em left dual object} of $X \in \mathcal{C}$ is an object $X^* \in \mathcal{C}$ equipped with morphisms $\eval_X : X^* \otimes X \to \unitobj$ and $\coev_X : \unitobj \to X \otimes X^*$ satisfying certain equations. A {\em right dual object} of $X \in \mathcal{C}$ is a left dual object of $X \in \mathcal{C}^{\rev}$. A rigid monoidal category is a monoidal category of which every object has a left dual object and a right dual object. We suppose that $\mathcal{C}$ is rigid. Then the assignment $X \mapsto X^*$ extends to a monoidal equivalence from $\mathcal{C}^{\op}$ to $\mathcal{C}^{\rev}$, which we call the left duality functor. We denote by ${}^*(-)$ a quasi-inverse of $(-)^*$. We may, and do, assume that $(-)^*$ and ${}^*(-)$ are strict monoidal and mutually inverse. Thus we have $(X \otimes Y)^* = Y^* \otimes X^*$, ${}^*(X^*) = X = ({}^*X)^*$, $\unitobj^* = \unitobj$, etc.

Given a left module category $\mathcal{M}$ over a monoidal category $\mathcal{C}$ (which is not necessarily rigid), we usually denote the action of $\mathcal{C}$ on $\mathcal{M}$ by $\catactl : \mathcal{C} \times \mathcal{M} \to \mathcal{M}$.
By an analogue of Mac Lane's coherence theorem, we may assume that all module categories over $\mathcal{C}$ are strict \cite[Remark 7.2.4]{MR3242743}.
Thus we have $\unitobj \catactl M = M$ and $(X \otimes Y) \catactl M = X \catactl (Y \catactl M)$ for $X, Y \in \mathcal{C}$ and $M \in \mathcal{M}$.

Let $\mathcal{M}$ and $\mathcal{N}$ be left $\mathcal{C}$-module categories.
A {\em lax left $\mathcal{C}$-module functor} \cite{2014arXiv1406.4204D}\footnote{This preprint has been published as \cite{MR3934626}, however, the definitions of a lax and an oplax $\mathcal{C}$-module functor are not included in the published version. We also recall some results on (op)lax module functors from the preprint version.} from $\mathcal{M}$ to $\mathcal{N}$ is a pair $(F, \xi)$ consisting of a functor $F : \mathcal{M} \to \mathcal{N}$ and a (not necessarily invertible) natural transformation
\begin{equation*}
  \xi_{X,M} : X \catactl F(M) \to F(X \catactl M)
  \quad (X \in \mathcal{C}, M \in \mathcal{M})
\end{equation*}
such that the equations $\xi_{\unitobj, M} = \id_M$ and $\xi_{X \otimes Y, M} = \xi_{X, Y \catactl M} \circ (\id_X \catactl \xi_{Y,M})$ hold for all objects $X, Y \in \mathcal{C}$ and $M \in \mathcal{M}$ ({\it cf}. \cite[Definition 7.2.1]{MR3242743}).
An {\em oplax left $\mathcal{C}$-module functor} from $\mathcal{M}$ to $\mathcal{N}$ is a pair $(F, \zeta)$ consisting of a functor $F: \mathcal{M} \to \mathcal{N}$ and a natural transformation
\begin{equation*}
  \zeta_{X,M} : F(X \catactl M) \to X \catactl F(M)
  \quad (X \in \mathcal{C}, M \in \mathcal{M})
\end{equation*}
such that, in a word, the pair $(F^{\op}, \zeta^{\op})$ is a lax $\mathcal{C}^{\op}$-module functor from $\mathcal{M}^{\op}$ to $\mathcal{N}^{\op}$.
A {\em strong left $\mathcal{C}$-module functor} is a lax (or oplax) left $\mathcal{C}$-module functor whose structure morphism is invertible.
We omit the definitions of morphisms lax, oplax and strong module functors; see \cite[Definition 2.7]{2014arXiv1406.4204D}.

Suppose that $\mathcal{C}$ is rigid. Then all lax $\mathcal{C}$-module functors and all oplax $\mathcal{C}$-module functors are strong \cite[Lemma 2.10]{2014arXiv1406.4204D}.
Thus we may simply call them $\mathcal{C}$-module functors.
Despite that we mainly consider the rigid case in this paper, the notions of lax and oplax $\mathcal{C}$-module functors are helpful when we deal with adjoints of module functors.

\subsection{Adjoints of module functors}

Given a functor $F$ admitting a right adjoint, we denote by $F^{\radj}$ a (fixed) right adjoint of $F$. If $F : \mathcal{L} \to \mathcal{M}$ and $G: \mathcal{M} \to \mathcal{N}$ are functors admitting right adjoints, then $F \circ G$ also has a right adjoint. Furthermore, there is a canonical isomorphism
\begin{equation}
  \label{eq:FG-r-adj}
  \Upsilon_{F,G} : (F \circ G)^{\radj} \to G^{\radj} \circ F^{\radj},
\end{equation}
which will be used frequently in this paper.
The canonical isomorphism \eqref{eq:FG-r-adj} is coherent in the following sense:
For three composable functors $F$, $G$ and $H$ admitting right adjoints, the diagram
\begin{equation*}
  \begin{tikzcd}[column sep = 64pt]
    (F \circ G \circ H)^{\radj}
    \arrow[r, "{\Upsilon_{F \circ G, H}}"]
    \arrow[d, "{\Upsilon_{F, G \circ H}}"']
    & H^{\radj} \circ (F \circ G)^{\radj}
    \arrow[d, "{H^{\radj} \circ \Upsilon_{F, G}}"] \\
    (G \circ H)^{\radj} \circ F^{\radj}
    \arrow[r, "{\Upsilon_{G, H} \circ F^{\radj}}"']
    & H^{\radj} \circ G^{\radj} \circ F^{\radj}
  \end{tikzcd}
\end{equation*}
is commutative. Here, we have used the symbol $\circ$ not only for the composition of functors but also for the horizontal composition of natural transformations.

We recall basic facts on adjoints of module functors from \cite{2014arXiv1406.4204D}.
We refer to a natural transformation making a functor $F$ an (op)lax left $\mathcal{C}$-module functor as an (op)lax left $\mathcal{C}$-module structure on $F$.
Let $\mathcal{M}$ and $\mathcal{N}$ be a left module categories over a monoidal category $\mathcal{C}$, and let $F: \mathcal{M} \to \mathcal{N}$ be a functor admitting a right adjoint. If $\zeta$ is an oplax left $\mathcal{C}$-module structure of $F$, then the natural transformation
\begin{align*}
  F^{\radj}((\id_X \catactl \varepsilon_N) \zeta_{X, F^{\radj}(N)}) \eta_{X \catactl F^{\radj}(N)}: X \catactl F^{\radj}(N) \to F^{\radj}(X \catactl N)
  \quad (X \in \mathcal{C}, N \in \mathcal{N})
\end{align*}
is a lax left $\mathcal{C}$-module structure of $F^{\radj}$, where $\eta$ and $\varepsilon$ are the unit and the counit of the adjunction $F \dashv F^{\radj}$, respectively.
This construction establishes a one-to-one correspondence between:
\begin{enumerate}
\item the class of oplax left $\mathcal{C}$-module structures of $F : \mathcal{M} \to \mathcal{N}$, and
\item the class of lax left $\mathcal{C}$-module structures of $F^{\radj} : \mathcal{N} \to \mathcal{M}$.
\end{enumerate}

We suppose that $F : \mathcal{M} \to \mathcal{N}$ and $G : \mathcal{L} \to \mathcal{M}$ are oplax $\mathcal{C}$-module functors between left $\mathcal{C}$-module categories and that both $F$ and $G$ have right adjoints. Then the functors $F^{\radj}$, $G^{\radj}$ and $(F \circ G)^{\radj}$ have lax left $\mathcal{C}$-module structures. The canonical isomorphism \eqref{eq:FG-r-adj} is in fact an isomorphism of lax $\mathcal{C}$-module functors.

A strong left $\mathcal{C}$-module functor can be thought of as either a lax or an oplax left $\mathcal{C}$-module functor.
If $F : \mathcal{M} \to \mathcal{N}$ is a strong left $\mathcal{C}$-module functor, then $F^{\radj} \circ F$ and $F \circ F^{\radj}$ are lax left $\mathcal{C}$-module functors. The unit and the counit of the adjunction $F \dashv F^{\radj}$ are morphisms of lax left $\mathcal{C}$-module functors.

\subsection{The internal Hom functor}
\label{subsec:closed-module-cat}

Let $\mathcal{C}$ be a monoidal category, and let $\mathcal{M}$ be a left $\mathcal{C}$-module category. The following notation will be used throughout this paper:

\begin{definition}
  For $M \in \mathcal{M}$, we define the functor $\T_M$ by
  \begin{equation*}
    \T_M : \mathcal{C} \to \mathcal{M},
    \quad \T_M(X) = X \catactl M \quad (X \in \mathcal{C}).
  \end{equation*}
\end{definition}

We say that $\mathcal{M}$ is {\em closed} if the functor $\T_M$ has a right adjoint for all $M \in \mathcal{M}$ ({\it cf}. the definition of a closed monoidal category).
Let $\mathcal{M}$ be a closed left $\mathcal{C}$-module category. Given an object $M \in \mathcal{M}$, we denote by $\iHom_{\mathcal{M}}(M, -)$ a (fixed) right adjoint of the functor $\T_M$. If $\mathcal{M}$ is clear from the context, then $\iHom_{\mathcal{M}}$ is also written as $\iHom$.
By definition, there is a natural isomorphism
\begin{equation}
  \label{eq:iHom-definition}
  \Hom_{\mathcal{M}}(X \catactl M, N)
  \cong \Hom_{\mathcal{C}}(X, \iHom_{\mathcal{M}}(M, N))
\end{equation}
for $X \in \mathcal{C}$ and $N \in \mathcal{M}$.
The assignment $(M, N) \mapsto \iHom(M, N)$ extends to a functor from $\mathcal{M}^{\op} \times \mathcal{M}$ to $\mathcal{C}$, called {\em the internal Hom functor}, in such a way that the isomorphism \eqref{eq:iHom-definition} is also natural in the variable $M$.
We denote by
\begin{equation*}
  \icoev_{M,X} : X \to \iHom(M, X \catactl M)
  \quad \text{and} \quad
  \ieval_{M,N} : \iHom(M, N) \catactl M \to N
\end{equation*}
the unit and the counit of the adjunction.
The morphisms $\icoev_{M,X}$ and $\ieval_{M,N}$ are dinatural in the variable $M \in \mathcal{M}$.

The class of objects of $\mathcal{M}$ becomes a $\mathcal{C}$-enriched category.
Indeed, for three objects $L, M, N \in \mathcal{M}$, the composition law
\begin{equation*}
  \icomp_{L, M, N} : \iHom(M, N) \otimes \iHom(L, M) \to \iHom(L, N)
\end{equation*}
is defined to be the morphism in $\mathcal{M}$ corresponding to
\begin{equation*}
  \ieval_{L, M, N}^{(2)} := \ieval_{M, N} \circ (\id_{\iHom(M, N)} \catactl \ieval_{L, M}):
  \iHom(M, N) \catactl \iHom(L, M) \catactl L \to N
\end{equation*}
under the adjunction isomorphism \eqref{eq:iHom-definition}.
The morphism $\icoev_{M, \unitobj}$ plays the role of the identity element with respect to the composition law.

According to \cite[Lemma 7.9.4]{MR3242743}, there are natural isomorphisms
\begin{equation*}
  X \otimes \iHom(M, N) \cong \iHom(M, X \catactl N),
  \quad \iHom(M, N) \otimes X^* \cong \iHom(X \catactl M, N)
\end{equation*}
for $M, N \in \mathcal{M}$ and $X \in \mathcal{C}$ when $\mathcal{C}$ is rigid. Below we prove some equations involving these natural isomorphisms.
For this purpose, we shall give constructions of these natural isomorphisms in a slightly more general setting than \cite{MR3242743}.

Let $\mathcal{M}$ be a left module category over a monoidal category $\mathcal{C}$, which is not necessarily rigid.
Since the functor $\T_M : \mathcal{C} \to \mathcal{M}$ for $M \in \mathcal{M}$ is a strong left $\mathcal{C}$-module functor, its right adjoint is a lax left $\mathcal{C}$-module functor. We denote the structure morphism of $\T_M^{\radj} = \iHom(M, -)$ by
\begin{equation}
  \label{eq:iHom-left-C-mod-str}
  \mathfrak{a}_{X,M,N} : X \otimes \iHom(M, N) \to \iHom(M, X \catactl N)
  \quad (X \in \mathcal{C}, N \in \mathcal{M}).
\end{equation}
The morphism $\mathfrak{a}_{X,M,N}$ is not invertible in general.
However, when $\mathcal{C}$ is rigid, it is invertible by the fact that every lax left $\mathcal{C}$-module functor is strong.

\begin{lemma}
  For $L, M, N \in \mathcal{M}$ and $X \in \mathcal{C}$, we have
  \begin{gather}
    \label{eq:iHom-left-C-mod-str-2}
    \mathfrak{a}_{X,M,N} = \iHom(M, \id_X \catactl \ieval_{M,N}) \circ \icoev_{M, X \otimes \iHom(M, N)}, \\
    \label{eq:iHom-left-C-mod-str-3}
    \ieval_{M, X \catactl N} \circ (\mathfrak{a}_{X,M,N} \catactl \id_M) = \id_X \catactl \ieval_{M,N}, \\
    \label{eq:iHom-composition-a-1}
    \icomp_{L,M,N} = \iHom(L, \ieval_{M,N}) \circ \mathfrak{a}_{\iHom(M, N), L, M}, \\
    \label{eq:iHom-composition-a-2}
    \icomp_{L,M, X \catactl N} \circ (\mathfrak{a}_{X,M,N} \otimes \id_{\iHom(L,M)})
    = \mathfrak{a}_{X, L, N} \circ (\id_X \otimes \icomp_{L,M,N}).
  \end{gather}
\end{lemma}
\begin{proof}
  Equation \eqref{eq:iHom-left-C-mod-str-2} follows from the construction of $\mathfrak{a}$.
  Also by the construction, the natural transformation $\ieval_{M,-}$ is a morphism of left $\mathcal{C}$-module functors. This implies \eqref{eq:iHom-left-C-mod-str-3}.
  
  Equation \eqref{eq:iHom-left-C-mod-str-2} implies that the right-hand side of \eqref{eq:iHom-composition-a-1} corresponds to the morphism $\ieval^{(2)}_{L,M,N}$ under the adjunction isomorphism \eqref{eq:iHom-definition}. Thus \eqref{eq:iHom-composition-a-1} follows from the definition of $\icomp_{L,M,N}$.
  Finally, we verify \eqref{eq:iHom-composition-a-2}. The left-hand side of \eqref{eq:iHom-composition-a-2} is computed as follows:
  \begin{align*}
    & \icomp_{L,M, X \catactl N} \circ (\mathfrak{a}_{X,M,N} \otimes \id_{\iHom(L,M)}) \\
    & = \iHom(L, \ieval_{M, X \catactl N}) \circ \mathfrak{a}_{\iHom(M, X \catactl N), L, M}
      \circ (\mathfrak{a}_{X,M,N} \otimes \id_{\iHom(L,M)}) \\
    & = \iHom(L, \ieval_{M, X \catactl N})
      \circ \iHom(L, \mathfrak{a}_{X, M, N} \catactl \id_M)
      \circ \mathfrak{a}_{X \otimes \iHom(M, N), L, M} \\
    & = \iHom(L, \id_X \catactl \ieval_{M,N}) \circ \mathfrak{a}_{X \otimes \iHom(M, N), L, M},
  \end{align*}
  where the first equation follows from \eqref{eq:iHom-composition-a-1}, the second from the naturality of $\mathfrak{a}$, and the last from \eqref{eq:iHom-left-C-mod-str-3}.
  On the other hand, the right-hand side of \eqref{eq:iHom-composition-a-2} is computed as follows:
  \begin{align*}
    & \mathfrak{a}_{X, L, N} \circ (\id_X \otimes \icomp_{L,M,N}) \\
    & = \mathfrak{a}_{X,L,N} \circ (\id_X \otimes \iHom(L, \ieval_{M,N})) \circ (\id_X \otimes \mathfrak{a}_{\iHom(M, N), L, M}) \\
    & = \iHom(L, \id_X \catactl \ieval_{M,N}) \circ \mathfrak{a}_{X,L, \iHom(M,N) \catactl M} \circ (\id_X \otimes \mathfrak{a}_{\iHom(M, N), L, M}) \\
    & = \iHom(L, \id_X \catactl \ieval_{M,N}) \circ \mathfrak{a}_{X \otimes \iHom(M, N), L, M},
  \end{align*}
  where the first equality follows from \eqref{eq:iHom-composition-a-1}, the second from the naturality of $\mathfrak{a}$, and the third from the axiom of a left $\mathcal{C}$-module functor. Hence \eqref{eq:iHom-composition-a-2} holds.
\end{proof}

Since $\mathcal{C}$ is a left $\mathcal{C}$-module category by the action given by the monoidal product, the functor $\T_X : \mathcal{C} \to \mathcal{C}$ is defined for each $X \in \mathcal{C}$. If $X \in \mathcal{C}$ is an object such that $\T_X$ has a right adjoint $[X, -] : \mathcal{C} \to \mathcal{C}$, then we have an isomorphism
\begin{equation}
  \label{eq:C-bimod-struc-of-iHom-0}
  [X, \iHom(M, -)]
  = (\T_X)^{\radj} \circ (\T_M)^{\radj}
  \mathop{\cong}^{\eqref{eq:FG-r-adj}}
  (\T_M \circ \T_X)^{\radj}
  = \iHom(X \catactl M, -)
\end{equation}
of lax left $\mathcal{C}$-module functors.

From now on, we assume that $\mathcal{C}$ is rigid.
For $X \in \mathcal{C}$, we can, and do, choose $\T_{X^*}$ as a right adjoint of the functor $\T_X$. We denote by
\begin{equation*}
  \mathfrak{b}_{M,-,X} : \iHom(M, -) \otimes X^*
  \to \iHom(X \catactl M, -)
\end{equation*}
the isomorphism \eqref{eq:C-bimod-struc-of-iHom-0}. By the construction of \eqref{eq:FG-r-adj}, we have
\begin{gather}
  \label{eq:C-bimod-struc-of-iHom-b}
  \mathfrak{b}_{M,N,X}
  = \iHom(X \catactl M, \ieval_{M,N} \circ (\id_{H} \catactl \eval_X \catactl \id_M)) \circ \icoev_{X \catactl M, H \otimes X^*}
\end{gather}
for $M, N \in \mathcal{M}$ and $X \in \mathcal{C}$, where $H = \iHom(M, N)$.
Since \eqref{eq:C-bimod-struc-of-iHom-0} is actually an isomorphism of left $\mathcal{C}$-module functors, we have
\begin{equation}
  \label{eq:C-bimod-struc-of-iHom-2}
  \mathfrak{a}_{X \catactl M, N} \circ (\id_X \otimes \mathfrak{b}_{M, N, Y})
  = \mathfrak{b}_{M, X \catactl N, Y} \circ (\mathfrak{a}_{X,M,N} \otimes \id_{Y^*})
\end{equation}
for $M, N \in \mathcal{M}$ and $X, Y \in \mathcal{C}$.

\begin{lemma}
  \label{lem:iHom-composition}
  For $L, M, N \in \mathcal{M}$ and $X \in \mathcal{C}$, we have
  \begin{gather}
    \label{eq:iHom-composition-b-0}
    \ieval_{X \catactl M, N} \circ (\mathfrak{b}_{M,N,X^*} \catactl \id_{X \catactl M})
    = \ieval_{M,N} \circ (\id_{\iHom(M,N)} \catactl \eval_X \catactl \id_M), \\
    \label{eq:iHom-composition-b-1}
    \mathfrak{b}_{M, N, \iHom(M, N)} \circ \coev_{\iHom(M, N)}
    = \iHom(\ieval_{M,N}, N) \circ \icoev_{N, \unitobj}, \\
    \label{eq:iHom-composition-b-2}
    \icomp_{X \catactl L,M,N} \circ (\id_{\iHom(M, N)} \otimes \mathfrak{b}_{L, M, X})
    = \mathfrak{b}_{L,N,X} \circ (\icomp_{L,M,N} \otimes \id_{X^*}).
  \end{gather}
\end{lemma}
\begin{proof}
  Equation \eqref{eq:C-bimod-struc-of-iHom-b} means that the morphism $\mathfrak{b}_{M,N,X}$ corresponds to the right hand side of \eqref{eq:iHom-composition-b-0} under the adjunction isomorphism \eqref{eq:iHom-definition}. Therefore \eqref{eq:iHom-composition-b-0} holds.
  Writing $H = \iHom(M, N)$ to save space, we verify \eqref{eq:iHom-composition-b-1} as follows:
  \begin{align*}
    & \mathfrak{b}_{M, N, H} \circ \coev_{H} \\
    & = \iHom(H \catactl M, \ieval_{M,N} \circ (\id_{H} \catactl \eval_H \catactl \id_M)) \circ \icoev_{H \catactl M, H \otimes H^*} \circ \coev_H \\
    & = \iHom(H \catactl M, \ieval_{M,N}) \circ \icoev_{H \catactl M, \unitobj}
      = \iHom(\ieval_{M,N}, N) \circ \icoev_{N, \unitobj},
  \end{align*}
  where the first equality follows from \eqref{eq:C-bimod-struc-of-iHom-b}, the second from the naturality of $\icoev_{H \catactl M, -}$ and the zig-zag equation for $\eval_H$ and $\coev_H$, and the last from the dinaturality of $\coev_{-,\unitobj}$.
  Finally, we prove \eqref{eq:iHom-composition-b-2} as follows:
  \begin{align*}
    & \icomp_{X \catactl L, M, N} \circ (\id_{H} \otimes \mathfrak{b}_{L, M, X}) \\
    & = \iHom(X \catactl L, \ieval_{M,N}) \circ \mathfrak{a}_{H, X \catactl L, M} \circ (\id_{H} \otimes \mathfrak{b}_{L, M, X}) \\
    & = \iHom({}^*\!X \catactl L, \ieval_{M,N})
      \circ \mathfrak{b}_{L, H \catactl M, X} \circ (\mathfrak{a}_{H, L, M} \otimes \id_{X^*}) \\
    & = \mathfrak{b}_{L, N, X}
      \circ (\iHom(L, \ieval_{M,N}) \otimes \id_{X^*})
      \circ (\mathfrak{a}_{H, L, M} \otimes \id_{X^*}) \\
    & = \mathfrak{b}_{L,N,X} \circ (\icomp_{L,M,N} \otimes \id_{X^*}),
  \end{align*}
  where the first and the last equality follow from \eqref{eq:iHom-composition-a-1}, the second from \eqref{eq:C-bimod-struc-of-iHom-2}, and the third from the naturality of $\mathfrak{b}$.
\end{proof}

We make the category $\mathcal{M}^{\op}$ a right $\mathcal{C}$-module category by the action
\begin{equation}
  \label{eq:M-op-action}
  M^{\op} \catactr X = ({}^* \! X \catactl M)^{\op}
  \quad (M \in \mathcal{M}, X \in \mathcal{C}).
\end{equation}
By the coherence property of \eqref{eq:FG-r-adj}, we see that the natural isomorphism $\mathfrak{b}_{-,N,-}$ makes $\iHom(-,N)$ a right $\mathcal{C}$-module functor from $\mathcal{M}^{\op}$ to $\mathcal{C}$.
Equation \eqref{eq:C-bimod-struc-of-iHom-2} means that $\mathfrak{a}$ and $\mathfrak{b}$ make $\iHom$ a $\mathcal{C}$-bimodule functor from $\mathcal{M}^{\op} \times \mathcal{M}$ to $\mathcal{C}$, where the left and the right action of $\mathcal{C}$ on the source are given by
\begin{equation}
  \label{eq:C-bimodule-M-op-x-M}
  X \catactl (M^{\op}, N) \catactr Y
  = (({}^* Y \catactl M)^{\op}, X \catactl N)
  \quad (X, Y \in \mathcal{C}, M, N \in \mathcal{M}).
\end{equation}

Finally, we note the following relation between $\icomp$, $\mathfrak{a}$ and $\mathfrak{b}$.

\begin{lemma}
  \label{lem:iHom-composition-2}
  For $L, M, N \in \mathcal{M}$ and $X \in \mathcal{C}$, we have
  \begin{gather}
    \label{eq:iHom-composition-ab}
    \begin{gathered}
      \icomp_{L, X \catactl M, N} \circ (\mathfrak{b}_{M,N,X} \otimes \mathfrak{a}_{X,L,M}) \\
      = \icomp_{L,M,N} \circ (\id_{\iHom(M,N)} \otimes \eval_X \otimes \id_{\iHom(L,M)}).
    \end{gathered}
  \end{gather}
\end{lemma}
\begin{proof}
  We compute the morphism corresponding to the left-hand side of \eqref{eq:iHom-composition-ab} under the adjunction isomorphism \eqref{eq:iHom-definition} as follows:
  \begin{align*}
    & \ieval_{L, N} ((\icomp_{L, X \catactl M, N} \circ (\mathfrak{b}_{M,N,X} \otimes \mathfrak{a}_{X,L,M})) \catactl \id_L) \\
    & = \ieval_{X \catactl M, N} \circ (\id_{\iHom(X \catactl M, N)} \catactl \ieval_{L, X \catactl N})
      \circ (\mathfrak{b}_{M,N,X} \catactl \mathfrak{a}_{X,L,M} \catactl \id_L) \\
    & = \ieval_{X \catactl M, N} \circ (\mathfrak{b}_{M,N,X} \catactl \id_X \catactl \ieval_{L,M}) \\
    & = \ieval_{M,N} \circ (\id_{\iHom(M,N)} \catactl \eval_X \catactl \ieval_{L,M}) \\
    & = \ieval_{L,M,N}^{(2)} \circ (\id_{\iHom(M,N)} \catactl \eval_X \catactl \id_{\iHom(L,M)} \catactl \id_L),
  \end{align*}
  where the first equality follows from the definition of $\icomp$,
  the second from \eqref{eq:iHom-left-C-mod-str-3},
  and the third from \eqref{eq:C-bimod-struc-of-iHom-b}.
  The result corresponds to the right-hand side of \eqref{eq:iHom-composition-ab} under \eqref{eq:iHom-definition}. Therefore \eqref{eq:iHom-composition-ab} holds. The proof is done.
\end{proof}

\subsection{Yoneda lemmas}

Throughout this subsection, we fix a category $\mathcal{L}$ and a closed left module category $\mathcal{M}$ over a monoidal category $\mathcal{C}$.
Given a functor $F: \mathcal{L} \to \mathcal{M}$, we define the functor $h_F$ by
\begin{equation*}
  h_F : \mathcal{M}^{\op} \times \mathcal{L} \to \mathcal{C},
  \quad (M^{\op}, L) \mapsto \iHom(M, F(L)).
\end{equation*}
In this subsection, we give Yoneda type lemmas for the internal Hom functor.

\begin{lemma}
  \label{lem:internal-Yoneda-0}
  Two functors $F$ and $G$ from $\mathcal{L}$ to $\mathcal{M}$ are isomorphic if and only if the functors $h_F$ and $h_G$ are isomorphic.
\end{lemma}
\begin{proof}
  The `only if' part is trivial. We prove the `if' part. Suppose that $h_F$ and $h_G$ are isomorphic. Then we have natural isomorphisms
  \begin{gather*}
    \Hom_{\mathcal{M}}(M, F(L))
    \cong \Hom_{\mathcal{C}}(\unitobj, \iHom(M, F(L))) \\
    \cong \Hom_{\mathcal{C}}(\unitobj, \iHom(M, G(L)))
    \cong \Hom_{\mathcal{M}}(M, G(L))
  \end{gather*}
  for $M \in \mathcal{M}$ and $L \in \mathcal{L}$. Thus, by the ordinary Yoneda lemma, we conclude that $F \cong G$ as functors.
\end{proof}

Given two functors $P$ and $Q$ with the same source and target, we denote by $\Nat(P, Q)$ the class of natural transformations from $P$ to $Q$.
Unlike the ordinary Yoneda lemma, Lemma~\ref{lem:internal-Yoneda-0} does not give a bijection between $\Nat(F, G)$ and $\Nat(h_F, h_G)$.

Now we assume that $\mathcal{C}$ is rigid.
Then $\mathcal{M}^{\op} \times \mathcal{L}$ is a right $\mathcal{C}$-module category by the action given in a similar way as \eqref{eq:C-bimodule-M-op-x-M}, and the functor $h_F$ is a right $\mathcal{C}$-module functor. If $\theta \in \Nat(F, G)$, then $\iHom(\id, \theta)$ is a morphism of right $\mathcal{C}$-module functors from $h_F$ to $h_G$.

\begin{lemma}
  \label{lem:internal-Yoneda-1}
  For functors $F, G : \mathcal{L} \to \mathcal{M}$, the map
  \begin{equation}
    \label{eq:internal-Yoneda-1-bijection}
    \Nat(F, G) \to \Nat_{\mathcal{C}}(h_F, h_G), \quad \theta \mapsto \iHom(\id, \theta)
  \end{equation}
  is bijective.
  Here, the target of the map \eqref{eq:internal-Yoneda-1-bijection} is the class of morphisms of right $\mathcal{C}$-module functors from $h_F$ to $h_G$.
\end{lemma}
\begin{proof}
  We denote the map \eqref{eq:internal-Yoneda-1-bijection} by $\Phi$.
  For $\xi \in \Nat_{\mathcal{C}}(h_F, h_G)$, we define $\overline{\Phi}(\xi)$ to be the unique natural transformation $\theta : F \to G$ such that the diagram
  \begin{equation*}
    \begin{tikzcd}[column sep = 32pt]
      \Hom_{\mathcal{M}}(M, F(L)) \arrow[r, "{\eqref{eq:iHom-definition}}"]
      \arrow[d, "{\Hom_{\mathcal{M}}(M, \theta_L)}"']
      & \Hom_{\mathcal{C}}(\unitobj, \iHom(M, F(L)))
      \arrow[d, "{\Hom_{\mathcal{C}}(\unitobj, \xi_{M, L})}"] \\
      \Hom_{\mathcal{M}}(M, G(L)) \arrow[r, "{\eqref{eq:iHom-definition}}"]
      & \Hom_{\mathcal{C}}(\unitobj, \iHom(M, G(L))) 
    \end{tikzcd}
  \end{equation*}
  is commutative for all $M \in \mathcal{M}$ and $L \in \mathcal{L}$. Explicitly,
  \begin{equation}
    \label{eq:internal-Yoneda-1-bijection-inverse}
    \overline{\Phi}(\xi)_L = \ieval_{F(L), G(L)} \circ (\xi_{F(L), L} \catactl \id_{F(L)}) \circ (\icoev_{F(L), \unitobj} \catactl \id_{F(L)}).
  \end{equation}
  It is obvious from the Yoneda lemma that $\overline{\Phi} \Phi$ is the identity map.
  We show that $\overline{\Phi} \Phi$ is also the identity map. We pick an element $\xi$ of the source of $\overline{\Phi}$ and set $\theta = \overline{\Phi}(\xi)$.
  We consider the diagram given as Figure~\ref{fig:proof-internal-Yoneda-1}.
  It is obvious that the top square of the diagram commutes.
  The middle one commutes by the assumption that $\xi$ is a morphism of right $\mathcal{C}$-module functors.
  The bottom one also commutes by the definition of $\theta$.
  Hence the diagram of Figure~\ref{fig:proof-internal-Yoneda-1} is commutative.
  By the definition of $\mathfrak{b}$, the composition along the left column is the adjunction isomorphism
  \begin{equation*}
    \Hom_{\mathcal{M}}(X \catactl M, F(L))
    \cong \Hom_{\mathcal{C}}(X, \iHom(M, F(L)),
  \end{equation*}
  and the same is said to the right column. By the Yoneda lemma, we obtain
  \begin{equation*}
    \xi_{M,L} = \iHom(\id_M, \theta_L) = \Phi \overline{\Phi}(\xi)_{M, L}
  \end{equation*}
  for all $M \in \mathcal{M}$ and $L \in \mathcal{L}$. The proof is done.
\end{proof}

\begin{figure}
  \centering
  \begin{equation*}
    \begin{tikzcd}[column sep = 60pt, row sep = 24pt]
      \Hom_{\mathcal{C}}(X, \iHom(M, F(L)))
      \arrow[r, "{\Hom_{\mathcal{C}}(X, \xi_{M,L})}"]
      \arrow[d, "{\text{adjunction}}"']
      & \Hom_{\mathcal{C}}(X, \iHom(M, G(L)))
      \arrow[d, "{\text{adjunction}}"] \\
      \Hom_{\mathcal{C}}(\unitobj, \iHom(M, F(L)) \otimes X^*)
      \arrow[r, "{\Hom_{\mathcal{C}}(\unitobj, \xi_{M,L} \otimes \id_{X^*})}" {yshift = 5pt}]
      \arrow[d, "{\Hom_{\mathcal{C}}(\unitobj, \mathfrak{b})}"']
      & \Hom_{\mathcal{C}}(\unitobj, \iHom(M, G(L)) \otimes X^*)
      \arrow[d, "{\Hom_{\mathcal{C}}(\unitobj, \mathfrak{b})}"] \\
      \Hom_{\mathcal{C}}(\unitobj, \iHom(X \catactl M, F(L)))
      \arrow[r, "{\Hom_{\mathcal{C}}(\unitobj, \xi_{X \catactl M, N})}"]
      \arrow[d, "{\text{adjunction}}"']
      & \Hom_{\mathcal{C}}(\unitobj, \iHom(X \catactl M, G(L)))
      \arrow[d, "{\text{adjunction}}"] \\
      \Hom_{\mathcal{C}}(X \catactl M, F(L))
      \arrow[r, "{\Hom_{\mathcal{C}}(X \catactl M, \theta_L)}"]
      & \Hom_{\mathcal{C}}(X \catactl M, G(L))
    \end{tikzcd}
  \end{equation*}
  \caption{Proof of Lemma~\ref{lem:internal-Yoneda-1}}
  \label{fig:proof-internal-Yoneda-1}
\end{figure}

We assume that $\mathcal{L}$ is a left $\mathcal{C}$-module category.
If $F: \mathcal{L} \to \mathcal{M}$ is a left $\mathcal{C}$-module functor, then the functor $h_F$ is a $\mathcal{C}$-bimodule functor from $\mathcal{M}^{\op} \times \mathcal{L}$ to $\mathcal{C}$, where the action of $\mathcal{C}$ on the source is given in a similar way as \eqref{eq:C-bimodule-M-op-x-M}.

\begin{lemma}[{\it cf}. {\cite[Lemma 4.11]{MR3435098}}]
  \label{lem:internal-Yoneda-2}
  The map
  \begin{equation}
    \label{eq:internal-Yoneda-2}
    {}_{\mathcal{C}}\Nat(F, G) \to {}_{\mathcal{C}}\Nat_{\mathcal{C}}(h_F, h_G),
    \quad \theta \mapsto \iHom(\id, \theta)
  \end{equation}
  is bijective for all left $\mathcal{C}$-module functors $F$ and $G$ from $\mathcal{L}$ to $\mathcal{M}$.
  Here, the source and the target of the map \eqref{eq:internal-Yoneda-2} are the class of morphisms of left $\mathcal{C}$-module functors from $F$ to $G$ and the class of morphisms of $\mathcal{C}$-bimodule functors from $h_F$ to $h_G$, respectively.
\end{lemma}
\begin{proof}
  Let $\Phi : \Nat(F, G) \to \Nat_{\mathcal{C}}(h_F, h_G)$ be the bijection of Lemma \ref{lem:internal-Yoneda-1}.
  It suffices to show that an element $\theta \in \Nat(F, G)$ is a morphism of right $\mathcal{C}$-module functor if and only if $\xi := \Phi(\theta)$ is a morphism of $\mathcal{C}$-bimodule functors.
  The `only if' part is obvious. We show the `if' part.
  We recall from the proof of Lemma \ref{lem:internal-Yoneda-1} that $\theta$ is expressed by $\xi$ as the right hand side of \eqref{eq:internal-Yoneda-1-bijection-inverse}. By using \eqref{eq:iHom-left-C-mod-str-3} and \eqref{eq:internal-Yoneda-1-bijection-inverse}, one can directly check that $\theta$ is a morphism of left $\mathcal{C}$-module functors. The proof is done.
\end{proof}

\section{Finite tensor categories and their modules}
\label{sec:ftc-and-modules}

\subsection{Finite module categories}

From now on till the end of this paper, we work over a field $\bfk$.
Given algebras $A$ and $B$ over $\bfk$, we denote by $\lmod{A}$, $\rmod{B}$ and $\bimod{A}{B}$ the category of finite-dimensional left $A$-modules, right $B$-modules and $A$-$B$-bimodules, respectively. We set $\Vect := \lmod{\bfk}$.
In this section, we summarize fundamental results on finite tensor categories and their modules that will be needed later. We first recall basic terminology:
\begin{enumerate}
\item A {\em finite abelian category} is a $\bfk$-linear category that is equivalent to $\rmod{A}$ as a $\bfk$-linear category for some finite-dimensional algebra $A$ over $\bfk$.
\item A {\em finite tensor category} is a finite abelian category $\mathcal{C}$ equipped with a structure of a rigid monoidal category such that the monoidal product $\otimes : \mathcal{C} \times \mathcal{C} \to \mathcal{C}$ is bilinear and the unit object $\unitobj \in \mathcal{C}$ is simple.
\item Let $\mathcal{C}$ be a finite tensor category. A {\em finite left $\mathcal{C}$-module category} is a finite abelian category $\mathcal{M}$ equipped with a structure of a left $\mathcal{C}$-module category such that the action $\catactl : \mathcal{C} \times \mathcal{M} \to \mathcal{M}$ is bilinear and right exact in each variable. A {\em finite right $\mathcal{C}$-module category} is defined analogously.
\item An {\em exact left $\mathcal{C}$-module category} is a finite left $\mathcal{C}$-module category $\mathcal{M}$ such that $P \catactl M$ is projective for all $M \in \mathcal{M}$ and all projective $P \in \mathcal{C}$.
\end{enumerate}
Unless otherwise noted, functors between finite abelian categories are implicitly assumed to be $\bfk$-linear.

\begin{example}
  \label{ex:Vect-action}
  Every finite abelian category $\mathcal{M}$ is a finite left $\Vect$-module category by the action $\catactl$ defined so that there is a natural isomorphism
  \begin{equation*}
    \Hom_{\mathcal{M}}(V \catactl M, N) \cong \Hom_{\bfk}(V, \Hom_{\mathcal{M}}(M, N))
    \quad (V \in \Vect, M, N \in \mathcal{M}).
  \end{equation*}
\end{example}

\begin{example}
  \label{ex:internal-Hom-C-A}
  Let $\mathcal{C}$ be a finite tensor category, and let $A$ be an algebra in $\mathcal{C}$. Then the category $\mathcal{C}_A$ of right $A$-modules in $\mathcal{C}$ is a finite left $\mathcal{C}$-module category.
  Moreover, $\mathcal{C}_A$ is closed. Indeed, the internal Hom functor is given by
  \begin{equation*}
    \iHom_A(M, N) = \Ker(\id_N \otimes \act_M^* - \act_N^{\natural} \otimes \id_{M^*} : N \otimes M^* \to N \otimes A^* \otimes M^*)
  \end{equation*}
  for $M, N \in \mathcal{C}_A$, where $\act_X : X \otimes A \to X$ for $X \in \mathcal{C}_A$ is the action of $A$ and
  \begin{equation*}
    \act_X^{\natural} := (\act_X \otimes \id_{A^*}) \circ (\id_X \otimes \coev_A) : X \to X \otimes A^*.
  \end{equation*}
\end{example}

\subsection{Morita theory in a finite tensor category}

Let $\mathcal{C}$ be a finite tensor category.
We say that two algebras $A$ and $B$ in $\mathcal{C}$ are {\em Morita equivalent} \cite[Definition 7.8.17]{MR3242743} if $\mathcal{C}_A \approx \mathcal{C}_B$ as left $\mathcal{C}$-module categories.
The algebras $A$ and $B$ are Morita equivalent if and only if ${}_A\mathcal{C} \approx {}_B\mathcal{C}$ as right $\mathcal{C}$-module categories, since, for every algebra $R$ in $\mathcal{C}$, there is an anti-equivalence between ${}_R\mathcal{C}$ and $\mathcal{C}_R$ induced by the duality functor of $\mathcal{C}$.

In this subsection, we recall basic results relevant to the Morita equivalence of algebras in $\mathcal{C}$.
Let $\mathcal{M}$ be a finite left $\mathcal{C}$-module category.
We note that $\mathcal{M}$ is closed by the finiteness of $\mathcal{M}$ and the right exactness of the action of $\mathcal{C}$ on $\mathcal{M}$.
Following to \cite[Definition 2.21]{2014arXiv1406.4204D}, we introduce:

\begin{definition}
  Given an object $M \in \mathcal{M}$, we say that:
  \begin{enumerate}
  \item $M$ is {\em $\mathcal{C}$-projective} if the functor $\iHom(M,-)$ is exact.
  \item $M$ is {\em $\mathcal{C}$-injective} if the functor $\iHom(-,M)$ is exact.
  \item $M$ is a {\em $\mathcal{C}$-generator} if the functor $\iHom(M, -)$ is faithful.
  \item $M$ is a {\em $\mathcal{C}$-progenerator} if it is a $\mathcal{C}$-projective $\mathcal{C}$-generator.
  \end{enumerate}
\end{definition}

We denote by $\relprj{\mathcal{C}}{\mathcal{M}}$ and $\relinj{\mathcal{C}}{\mathcal{M}}$ the full subcategory of $\mathcal{M}$ consisting of all $\mathcal{C}$-projective objects and all $\mathcal{C}$-injective objects, respectively.
The category $\mathcal{M}^{\op}$ is a finite left $\mathcal{C}^{\op}$-module category by the action given by $X^{\op} \catactl M^{\op} = (X \catactl M)^{\op}$ for $X \in \mathcal{C}$ and $M \in \mathcal{M}$ (here, unlike Subsection~\ref{subsec:closed-module-cat}, we do not regard $\mathcal{M}^{\op}$ as a right $\mathcal{C}$-module category). Since there are natural isomorphisms
\begin{gather*}
  \Hom_{\mathcal{M}^{\op}}(X^{\op} \catactl M^{\op}, N^{\op})
  = \Hom_{\mathcal{M}}(N, X \catactl M)
  \cong \Hom_{\mathcal{M}}(X^* \catactl N, M) \\
  \cong \Hom_{\mathcal{C}}(X^*, \iHom(N, M))
  \cong \Hom_{\mathcal{C}^{\op}}(X^{\op}, ({}^*\iHom(N, M))^{\op})
\end{gather*}
for $M, N \in \mathcal{M}$ and $X \in \mathcal{C}$, we have
\begin{equation*}
  \iHom_{\mathcal{M}^{\op}}(M^{\op}, N^{\op}) = ({}^*\iHom_{\mathcal{M}}(N, M))^{\op}
  \quad (M, N \in \mathcal{M})
\end{equation*}
in $\mathcal{C}^{\op}$. Hence we may identify
\begin{equation}
  \label{eq:relative-proj-opposite}
  \relinj{\mathcal{C}^{\op}}{\mathcal{M}^{\op}}
  = \relprj{\mathcal{C}}{\mathcal{M}}
  \quad \text{and} \quad
  \relprj{\mathcal{C}^{\op}}{\mathcal{M}^{\op}}
  = \relinj{\mathcal{C}}{\mathcal{M}}.
\end{equation}

Given an abelian category $\mathcal{A}$, we denote by $\relprj{}{\mathcal{A}}$ and $\relinj{}{\mathcal{A}}$ the full subcategory of $\mathcal{A}$ consisting of all projective and injective objects of $\mathcal{A}$, respectively.
If we view a finite abelian category $\mathcal{M}$ as a finite left $\Vect$-module category as in Example~\ref{ex:Vect-action}, then the internal Hom functor of $\mathcal{M}$ coincides with the ordinary Hom functor and therefore we have
$\relprj{}{\mathcal{M}} = \relprj{\Vect}{\mathcal{M}}$ and
$\relinj{}{\mathcal{M}} = \relinj{\Vect}{\mathcal{M}}$.

Now let $\mathcal{M}$ be a finite left $\mathcal{C}$-module category. Then we have:

\begin{lemma}
  \label{lem:relative-proj-1}
  $\relprj{}{\mathcal{M}} \subset \relprj{\mathcal{C}}{\mathcal{M}}$,
  $\relinj{}{\mathcal{M}} \subset \relinj{\mathcal{C}}{\mathcal{M}}$.
\end{lemma}
\begin{proof}
  The first inclusion is found in the proof of \cite[Lemma 2.25]{2014arXiv1406.4204D}.
  The second one is proved in a similar way. One may also obtain the second one by applying the first one to $\mathcal{M}^{\op}$ in view of \eqref{eq:relative-proj-opposite}.
\end{proof}

\begin{lemma}
  \label{lem:relative-proj-2}
  $\relprj{\mathcal{C}}{\mathcal{M}}$ and $\relinj{\mathcal{C}}{\mathcal{M}}$ are closed under the action of $\mathcal{C}$.
\end{lemma}
\begin{proof}
  If $X \in \mathcal{C}$ and $P \in \relprj{\mathcal{C}}{\mathcal{M}}$, then the functor $\iHom(X \catactl P, -)$ is exact since it is isomorphic to $\iHom(P, -) \otimes X^*$, and therefore $X \catactl P$ belongs to $\relprj{\mathcal{C}}{\mathcal{M}}$.
  The case of $\relinj{\mathcal{C}}{\mathcal{M}}$ is proved in a similar way.
\end{proof}

We further discuss $\mathcal{C}$-projective and $\mathcal{C}$-injective objects in Subsection~\ref{subsec:C-proj-C-inj}.

Regarding $\mathcal{C}$-generators, the following characterization is known:

\begin{lemma}[{\cite[Lemma 2.22]{2014arXiv1406.4204D}}]
  An object $G \in \mathcal{M}$ is a $\mathcal{C}$-generator if and only if every object of $\mathcal{M}$ is a quotient of $X \catactl G$ for some $X \in \mathcal{C}$.
  In particular, a generator of $\mathcal{M}$ is a $\mathcal{C}$-generator.
\end{lemma}

For each object $M \in \mathcal{M}$, there is a left $\mathcal{C}$-module functor
\begin{equation*}
  F_M : \mathcal{M} \to \mathcal{C}_{\iEnd(M)},
  \quad X \mapsto \iHom(M, X),
\end{equation*}
where the algebra $\iEnd(M) := \iHom(M, M)$ acts on $\iHom(M, X)$ by the composition law for the internal Hom functor.
The functor $F_M$ is an equivalence if and only if $M$ is a $\mathcal{C}$-progenerator \cite[Theorem 2.24]{2014arXiv1406.4204D}.
An important consequence is that, since a projective generator of $\mathcal{M}$ is a $\mathcal{C}$-progenerator, $\mathcal{M}$ is equivalent to $\mathcal{C}_A$ for some algebra $A$ in $\mathcal{C}$ \cite[Theorem 2.18]{2014arXiv1406.4204D}.

By considering the case where $\mathcal{M} = \mathcal{C}_B$ for some algebra $B$ in $\mathcal{C}$, we obtain the following characterization of Morita equivalence: Two algebras $A$ and $B$ in $\mathcal{C}$ are Morita equivalent if and only if there is a $\mathcal{C}$-progenerator $P \in \mathcal{C}_B$ such that $A$ is isomorphic to $\iEnd(P)$ as an algebra in $\mathcal{C}$ 

\subsection{Eilenberg-Watts equivalence}

The Eilenberg-Watts theorem states that a functor between categories of modules is given by tensoring a bimodule if and only if that functor has a right adjoint.
An analogous result is known for algebras in finite tensor categories. Given algebras $A$ and $B$ in a finite tensor category $\mathcal{C}$, we denote by ${}_A \mathcal{C}_B$ the category of $A$-$B$-bimodules in $\mathcal{C}$. There is a functor
\begin{equation}
  \label{eq:EW-equivalence-in-C}
  {}_A\mathcal{C}_B \to \Rex_{\mathcal{C}}(\mathcal{C}_A, \mathcal{C}_B),
  \quad M \mapsto (-) \otimes_A M,
\end{equation}
where $\Rex_{\mathcal{C}}(\mathcal{M}, \mathcal{N})$ for finite left $\mathcal{C}$-module categories $\mathcal{M}$ and $\mathcal{N}$ is the category of left $\mathcal{C}$-module right exact functors from $\mathcal{M}$ to $\mathcal{N}$.
The functor \eqref{eq:EW-equivalence-in-C} is an equivalence \cite[Proposition 7.11.1]{MR3242743}. A quasi-inverse of \eqref{eq:EW-equivalence-in-C} is established by sending an object $F$ of the target category to the right $B$-module $F(A)$ made into an $A$-$B$-bimodule in $\mathcal{C}$ by the left action of $A$ given by
\begin{equation*}
  A \otimes F(A)
  \xrightarrow{\quad f_{A,A} \quad} F(A \otimes A) \xrightarrow{\quad F(m) \quad} F(A),
\end{equation*}
where $f$ is the left $\mathcal{C}$-module structure of $F$ and $m : A \otimes A \to A$ is the multiplication of the algebra $A$.
We call \eqref{eq:EW-equivalence-in-C} the {\em Eilenberg-Watts equivalence}.

The left $\mathcal{C}$-module category $\mathcal{C}_B$ has the internal Hom functor, which we denote by $\iHom_B$ (see Example~\ref{ex:internal-Hom-C-A}).
If $M \in {}_A\mathcal{C}_B$ and $X \in \mathcal{C}_B$, then $\iHom_B(M, X)$ is a right $A$-module in $\mathcal{C}$ by the action
\begin{equation*}
  \iHom_B(\act_M^{\natural}, \id_X) \circ \mathfrak{b}_{M,X,{}^*\!A} : \iHom_B(M, X) \otimes A \to \iHom_B(M, X),
\end{equation*}
where $\act_M^{\natural}$ is given in Example~\ref{ex:internal-Hom-C-A}. The functor $\iHom_B(M, -) : \mathcal{C}_B \to \mathcal{C}_A$ is right adjoint to the functor $(-) \otimes_A M : \mathcal{C}_A \to \mathcal{C}_B$.
By composing \eqref{eq:EW-equivalence-in-C} and the contravariant functor $F \mapsto F^{\radj}$, we obtain an anti-equivalence
\begin{equation}
  \label{eq:EW-equivalence-in-C-lex}
  {}_A\mathcal{C}_B \to \Lex_{\mathcal{C}}(\mathcal{C}_B, \mathcal{C}_A)
  \quad M \mapsto \iHom_B(M, -),
\end{equation}
where the target is a left exact variant of $\Rex_{\mathcal{C}}(\mathcal{M}, \mathcal{N})$.

We note some consequences of the equivalences \eqref{eq:EW-equivalence-in-C} and \eqref{eq:EW-equivalence-in-C-lex}. We recall from \cite[Lemma 2.4.13]{MR4254952} that the duality functors induce anti-equivalences
\begin{equation*}
  (-)^* : {}_A\mathcal{C}_B \to {}_{B^{**}}\mathcal{C}_A
  \quad \text{and} \quad
  {}^*(-) : {}_A\mathcal{C}_B \to {}_B\mathcal{C}_{{}^{**}\!A}
\end{equation*}
of categories. If $M \in {}_A \mathcal{C}_B$, then ${}^*\iHom_A(-, M^*) : \mathcal{C}_A \to \mathcal{C}_B$ is a left $\mathcal{C}$-module right exact functor. Since ${}^*\iHom_A(A, M^*) \cong {}^*(M^*) \cong M$ as $A$-$B$-bimodules in $\mathcal{C}$, the equivalence \eqref{eq:EW-equivalence-in-C} yields the following formula \cite[Example 7.9.8]{MR3242743} of the internal Hom functor:

\begin{lemma}
  \label{lem:iHom-dual}
  For $M \in {}_A \mathcal{C}_B$, there is an isomorphism
  \begin{equation}
    {}^*\iHom_A(-, M^*) \cong (-) \otimes_A M
  \end{equation}
  of left $\mathcal{C}$-module functors from $\mathcal{C}_A$ to $\mathcal{C}_B$. In particular, we have
  \begin{equation*}
    \iHom_A(X, Y) \cong (X \otimes_A {}^*Y)^*
    \quad (X, Y \in \mathcal{C}_A).
  \end{equation*}
\end{lemma}

This lemma implies that $A^* \in \mathcal{C}_A$ is $\mathcal{C}$-injective. Indeed,
\begin{equation*}
  \iHom_A(M, A^*) \cong (M \otimes_A A)^* \cong M^* \quad (M \in \mathcal{C}_A).
\end{equation*}

Let $\mathcal{M}$ be a finite left $\mathcal{C}$-module category.
We say that a functor $F : \mathcal{M} \to \mathcal{C}$ is {\em $\mathcal{C}$-representable} if $F \cong \iHom(M, -)$ for some object $M \in \mathcal{M}$. The $\mathcal{C}$-representability of a contravariant functor is defined analogously.
We note the following Lemmas~\ref{lem:C-representability-1} and \ref{lem:C-representability-2} on the $\mathcal{C}$-representability:

\begin{lemma}
  \label{lem:C-representability-1}
  Let $\mathcal{M}$ be as above.
  A left $\mathcal{C}$-module functor $F : \mathcal{M} \to \mathcal{C}$ is $\mathcal{C}$-representable if and only if it is left exact.
\end{lemma}
\begin{proof}
  We may assume that $\mathcal{M} = \mathcal{C}_B$ for some algebra $B$ in $\mathcal{C}$. If $F$ is left exact, then it is $\mathcal{C}$-representable by the anti-equivalence \eqref{eq:EW-equivalence-in-C-lex} with $A = \unitobj$. The `only if' part is obvious from the basic properties of the internal Hom functor.
\end{proof}

\begin{lemma}
  \label{lem:C-representability-2}
  Let $\mathcal{M}$ be as above, and make $\mathcal{M}^{\op}$ a right $\mathcal{C}$-module category by the action \eqref{eq:M-op-action}.
  A right $\mathcal{C}$-module functor $F : \mathcal{M}^{\op} \to \mathcal{C}$ is $\mathcal{C}$-representable if and only if it is left exact.
\end{lemma}
\begin{proof}
  We may assume that $\mathcal{M} = \mathcal{C}_A$ for some algebra $A$ in $\mathcal{C}$.
  We suppose that the functor $F : \mathcal{M}^{\op} \to \mathcal{C}$ is left exact.
  By our convention \eqref{eq:M-op-action} on the action, the functor $G: \mathcal{M} \to \mathcal{C}$ defined by $G(X) = {}^*F(X)$ for $X \in \mathcal{M}$ is a right exact left $\mathcal{C}$-module functor.
  Hence, by the equivalence \eqref{eq:EW-equivalence-in-C} with $B = \unitobj$, there is a left $A$-module $M$ such that $G \cong (-) \otimes_A M$. By Lemma~\ref{lem:iHom-dual}, we have
  \begin{equation*}
    F(X) \cong G(X)^* \cong (X \otimes_A M)^* \cong \iHom_A(X, M^*)
  \end{equation*}
  for $X \in \mathcal{C}_A$. Namely, $F$ is $\mathcal{C}$-representable.
  The converse is obvious.
\end{proof}

\subsection{Nakayama functor}
\label{subsec:nakayama}

We close this section by briefly reviewing the theory of {\em Nakayama functors} and their applications established in \cite{MR4042867}.
The Nakayama functor for a finite-dimensional algebra $A$ is the endofunctor on $\rmod{A}$ given by $M \mapsto M \otimes_{A} A^*$, where $A^* = \Hom_{\bfk}(A, \bfk)$ is the dual space of $A$ regarded as an $A$-bimodule in a natural way. Fuchs, Schaumann and Schweigert \cite{MR4042867} pointed out that the Nakayama functor is expressed by
\begin{equation}
  \label{eq:Nakayama-coend-formula}
  M \otimes_{A} A^* = \int^{X \in \rmod{A}} \Hom_{A}(M, X)^* \otimes_{\bfk} X
  \quad (M \in \rmod{A}),
\end{equation}
where the integral means a coend \cite{MR1712872}. Thus we can define the Nakayama functor $\Nak_{\mathcal{M}}: \mathcal{M} \to \mathcal{M}$ for any finite abelian category $\mathcal{M}$ as an endofunctor on $\mathcal{M}$ expressed by the same coend \cite[Definition 3.14]{MR4042867}.
By rephrasing known results in the representation theory of finite-dimensional algebras, we see that the following assertions for $\mathcal{M}$ are equivalent:
\begin{enumerate}
\item The Nakayama functor $\Nak_{\mathcal{M}}$ is an equivalence.
\item All projective objects of $\mathcal{M}$ are injective.
\item All injective objects of $\mathcal{M}$ are projective.
\item $\mathcal{M} \approx \rmod{A}$ for some Frobenius algebra $A$ over $\bfk$.
\end{enumerate}

Just for describing an important feature of the Nakayama functor, we introduce the following terminology: A functor $F$ between finite abelian categories is {\em admissible} if $F^{\rradj} := (F^{\radj})^{\radj}$ exists and has a right adjoint.
Let $\mathcal{M}$ and $\mathcal{N}$ be finite abelian categories.
If $F: \mathcal{M} \to \mathcal{N}$ is an admissible functor, then the universal property of the Nakayama functor yields an isomorphism
\begin{equation}
  \label{eq:Nakayama-cano-iso}
  \mathfrak{n}_F : F^{\rradj} \circ \Nak_{\mathcal{M}}
  \to \Nak_{\mathcal{N}} \circ F
\end{equation}
of functors. The isomorphism $\mathfrak{n}_F$ is natural in $F$ and coherent in a certain sense; see \cite[Theorem 3.18]{MR4042867}.

Given an even integer $n$ and a left module category $\mathcal{M}$ over a rigid monoidal category $\mathcal{C}$, we denote by ${}_{(n)}\mathcal{M}$ the category $\mathcal{M}$ viewed as a left $\mathcal{C}$-module category by the new action $\tilde{\catactl}$ defined by $X \mathbin{\tilde{\catactl}} M = D^{n}(X) \catactl M$ for $X \in \mathcal{C}$ and $M \in \mathcal{M}$, where $\catactl$ is the original action and $D = (-)^*$. We often say that $F: \mathcal{M} \to \mathcal{M}'$ is a {\em twisted left $\mathcal{C}$-module functor} if $F$ is a left $\mathcal{C}$-module functor from ${}_{(n)}\mathcal{M}$ to ${}_{(n')}\mathcal{M}'$ for some even integers $n$ and $n'$.

The canonical isomorphism \eqref{eq:Nakayama-cano-iso} makes the Nakayama functor a twisted module functor. More precisely, if $\mathcal{M}$ is a finite left module category over a finite tensor category $\mathcal{C}$, then the Nakayama functor $\Nak_{\mathcal{M}}$ is a left $\mathcal{C}$-module functor from $\mathcal{M}$ to ${}_{(-2)}\mathcal{M}$ with the structure morphism
\begin{equation}
  \label{eq:Nakayama-tw-module-structure}
  {}^{**} \! X \catactl \Nak_{\mathcal{M}}(M) \to \Nak_{\mathcal{M}}(X \catactl M)
  \quad (X \in \mathcal{C}, M \in \mathcal{M})
\end{equation}
given by \eqref{eq:Nakayama-cano-iso} with $F = X \catactl (-)$.
We also note that \eqref{eq:Nakayama-cano-iso} is an isomorphism of twisted left $\mathcal{C}$-module functors if $F : \mathcal{M} \to \mathcal{N}$ is an admissible left $\mathcal{C}$-module functor between finite left $\mathcal{C}$-module categories.

Let $\mathcal{C}$ be a finite tensor category. By \eqref{eq:Nakayama-tw-module-structure} with $\mathcal{M} = \mathcal{C}$ and $M = \unitobj$, we obtain a natural isomorphism $\Nak_{\mathcal{C}}(X) \cong {}^{**}\!X \otimes \Nak_{\mathcal{C}}(\unitobj)$ for $X \in \mathcal{C}$.
Since all projective objects of $\mathcal{C}$ are injective \cite[Proposition 6.1.3]{MR3242743}, $\Nak_{\mathcal{C}}$ is an equivalence.
Hence the object $\Nak_{\mathcal{C}}(\unitobj)$ is invertible with respect to the tensor product.

\begin{definition}
  \label{def:Radford-isomorphism}
  For $X \in \mathcal{C}$, we define the {\em Radford isomorphism} by
  \begin{equation*}
    \newcommand{\xarr}[1]{\xrightarrow{\makebox[10em]{$\scriptstyle #1$}}}
    \begin{aligned}
      \mathfrak{r}_X := \Big(
      {}^{**}X \otimes \Nak_{\mathcal{C}}(\unitobj)
      & \xarr{\text{$\mathfrak{n}_F$ with $F = X \otimes \id_{\mathcal{C}}$}}
      \Nak_{\mathcal{C}}(X) \\
      & \xarr{\text{$\mathfrak{n}_F^{-1}$ with $F = \id_{\mathcal{C}} \otimes X$}}
      \Nak_{\mathcal{C}}(\unitobj) \otimes X^{**} \Big).
    \end{aligned}
  \end{equation*}
\end{definition}

Remarkably, as explained in \cite{MR4042867}, the Radford $S^4$-formula for finite tensor categories \cite[Theorem 7.19.1]{MR3242743} is obtained as an application of the Nakayama functor.
Indeed, since $\Nak_{\mathcal{C}}(\unitobj)$ is invertible, we have a natural isomorphism
\begin{equation*}
  X^{****} \cong \Nak_{\mathcal{C}}(\unitobj)^* \otimes X \otimes \Nak_{\mathcal{C}}(\unitobj)
  \quad (X \in \mathcal{C})
\end{equation*}
induced by the Radford isomorphism.

\section{Relative Serre functor}
\label{sec:rel-Serre}

\subsection{Definition and basic properties}

Let $\mathcal{C}$ be a finite tensor category, and let $\mathcal{M}$ be a finite left $\mathcal{C}$-module category.
We suppose that $\mathcal{M}$ is exact.
It is known that the internal Hom functor of $\mathcal{M}$ is exact in each variable in this case \cite[Corollary 7.9.6]{MR3242743}.
The relative Serre functor \cite{MR3435098,MR4042867,2019arXiv190400376S} of $\mathcal{M}$ is defined to be an endofunctor $\Ser$ on $\mathcal{M}$ such that there is a natural isomorphism
\begin{equation}
  \label{eq:rel-Serre-def-iso-exact}
  \iHom(M, N)^* \cong \iHom(N, \Ser(M))
  \quad (M, N \in \mathcal{M}).
\end{equation}

This definition cannot be used in the case where $\mathcal{M}$ is not exact. Indeed, since the right-hand side of \eqref{eq:rel-Serre-def-iso-exact} is left exact in the variable $N \in \mathcal{M}$, the existence of a natural isomorphism \eqref{eq:rel-Serre-def-iso-exact} forces all objects $M \in \mathcal{M}$ to be $\mathcal{C}$-projective. If every object of $\mathcal{M}$ is $\mathcal{C}$-projective, then $\mathcal{M}$ is exact (see Lemma~\ref{lem:C-proj}).
For the case where $\mathcal{M}$ is not necessarily exact, we modify the above definition as follows:

\begin{definition}
  Let $\mathcal{C}$ be a finite tensor category, and let $\mathcal{M}$ be a finite left $\mathcal{C}$-module category.
  A {\em relative Serre functor} of $\mathcal{M}$ is a right exact endofunctor $\Ser$ on $\mathcal{M}$ such that there is a natural isomorphism
  \begin{equation}
    \label{eq:rel-Serre-def-iso}
    \iHom(P, M)^* \cong \iHom(M, \Ser(P))
    \quad (P \in \relprj{\mathcal{C}}{\mathcal{M}}, M \in \mathcal{M}).
  \end{equation}
\end{definition}

If $\mathcal{M}$ is exact, then the above definition agrees with that used in \cite{MR3435098,MR4042867,2019arXiv190400376S}.
For a finite abelian category $\mathcal{M}$, there is a natural isomorphism
\begin{equation}
  \label{eq:Nakayama-Hom-dual}
  \Hom_{\mathcal{M}}(P, M)^* \cong \Hom_{\mathcal{M}}(M, \Nak_{\mathcal{M}}(P))
\end{equation}
for $M \in \mathcal{M}$ and $P \in \relprj{}{\mathcal{M}}$; see, {\it e.g.}, \cite{2021arXiv210313702S}. This means that $\Nak_{\mathcal{M}}$ is a relative Serre functor of the $\Vect$-module category $\mathcal{M}$.

The aim of this section is to extend some of known results of the relative Serre functor to a finite module category which is not necessarily exact.
Let $\mathcal{C}$ be a finite tensor category, and let $\mathcal{M}$ be a finite left $\mathcal{C}$-module category. Since $\mathcal{C}$ is rigid and $\mathcal{M}$ is closed, there is a natural isomorphism
\begin{equation*}
  \mathfrak{c}_{X,M,N,Y} : X \otimes \iHom(M, N) \otimes Y^* \to \iHom(Y \catactl M, X \catactl N)
\end{equation*}
defined by the both sides of \eqref{eq:C-bimod-struc-of-iHom-2}.
The natural isomorphism
\begin{equation*}
  (\mathfrak{c}_{{}^{**}\!X,M,N,{}^{*}Y}^{-1})^*:
  X \otimes \iHom(M, N)^* \otimes Y
  \to \iHom({}^{**}\!X \catactl M, {}^{*}Y \catactl N)^*
\end{equation*}
makes the functor $\iHom(-, -)^*$ a $\mathcal{C}$-bimodule functor from $\mathcal{M}^{\op} \times {}_{(-2)}\mathcal{M}$ to $\mathcal{C}$ (see Subsection \ref{subsec:closed-module-cat} for our convention on the right action of $\mathcal{C}$ on $\mathcal{M}^{\op}$).
Our results in this section are summarized as follows:

\begin{theorem}
  \label{thm:rel-Serre-summary}
  Let $\mathcal{C}$ be a finite tensor category, and let  $\mathcal{M}$ be a finite left $\mathcal{C}$-module category.
  A relative Serre functor of $\mathcal{M}$ exists and is unique up to isomorphism as a functor.
  Let $\Ser$ be a relative Serre functor of $\mathcal{M}$. Then we have:
  \begin{enumerate}
  \item \label{item:rel-Serre-main-1}
    The functor $\Ser : \mathcal{M} \to \mathcal{M}$ has a twisted left $\mathcal{C}$-module structure
    \begin{equation*}
      \mathfrak{s}_{X,M} : X^{**} \catactl \Ser(M) \to \Ser(X \catactl M)
      \quad (X \in \mathcal{C}, M \in \mathcal{M})
    \end{equation*}
    such that there is an isomorphism
    \begin{equation*}
      \iHom(P, M)^* \cong \iHom(M, \Ser(P))
      \quad (M \in \mathcal{M}, P \in \relprj{\mathcal{C}}{\mathcal{M}})
    \end{equation*}
    of $\mathcal{C}$-bimodule functors from $\mathcal{M}^{\op} \times {}_{(-2)}\relprj{\mathcal{C}}{\mathcal{M}}$ to $\mathcal{C}$.
  \item \label{item:rel-Serre-main-2}
    There is an isomorphism $\Nak_{\mathcal{C}}(\unitobj) \catactl \Ser \cong \Nak_{\mathcal{M}}$ of twisted left $\mathcal{C}$-module functors, where the twisted left $\mathcal{C}$-module structure of the left-hand side is given by composing the Radford isomorphism of Definition~\ref{def:Radford-isomorphism} and the twisted left $\mathcal{C}$-module structure $\mathfrak{s}$ of Part (\ref{item:rel-Serre-main-1}).
  \item \label{item:rel-Serre-main-3}
    The restriction of $\Ser$ induces an equivalence from $\relprj{\mathcal{C}}{\mathcal{M}}$ to $\relinj{\mathcal{C}}{\mathcal{M}}$ with quasi-inverse given by the restriction of a right adjoint of $\Ser$.
  \end{enumerate}
\end{theorem}

The existence and the uniqueness part and Parts (\ref{item:rel-Serre-main-1}), (\ref{item:rel-Serre-main-2}) and (\ref{item:rel-Serre-main-3}) of Theorem~\ref{thm:rel-Serre-summary} are separately proved in Subsections \ref{subsec:rel-Serre-existence}, \ref{subsec:rel-Serre-twisted-module}, \ref{subsec:rel-Serre-Nakayama} and \ref{subsec:rel-Serre-equiv-PCM-ICM}, respectively.
In view of Parts (1) and (2), one might wonder if a relative Serre functor of $\mathcal{M}$ is unique up to isomorphism as a twisted left $\mathcal{C}$-module functor. It is true. A precise formulation of this claim and related results are given in Subsection \ref{subsec:rel-Serre-uniqueness}.

For our applications of relative Serre functors to quasi-Frobenius algebras discussed in the next section, the case where $\mathcal{M} = \mathcal{C}_A$ for some algebra $A$ in $\mathcal{C}$ is especially important. A formula of a relative Serre functor for this case will be given by Lemma \ref{lem:rel-Serre-mod-A} as a part of the proof of the existence of a relative Serre functor of a finite left module category. In Subsection~\ref{subsec:rel-Serre-A-mod}, we will have a further discussion on a relative Serre functor of $\mathcal{C}_A$ and its structure morphisms.

The Nakayama functor has the universal property \eqref{eq:Nakayama-coend-formula}. In view of similarity between \eqref{eq:rel-Serre-def-iso} and \eqref{eq:Nakayama-Hom-dual}, one might expect that a relative Serre functor $\Ser$ of a finite left $\mathcal{C}$-module category $\mathcal{M}$ is also characterized by a universal property.
In a forthcoming work, we will show that $\Ser$ is expressed as
\begin{equation*}
  \Ser(M) = \oint^{X} \iHom(M, X)^* \catactl X
  \quad (M \in \mathcal{M}),
\end{equation*}
where the contour integral symbol is used to mean the {\em module coend} introduced by Bortolussi and Mombelli in \cite{MR4226562}. This is a common generalization of the universal property \eqref{eq:Nakayama-coend-formula} of the Nakayama functor and the formula \cite[Theorem 4.14]{MR4226562} of the relative Serre functor of an exact module category.

\begin{remark}
  A large part of Theorem~\ref{thm:rel-Serre-summary} will be proved in an abstract categorical setting in the next subsection.
  Thus it is natural to expect a generalization of our results to a broader class of monoidal categories and their modules. Grothendieck-Verdier (GV) categories \cite{MR3134025} are a class of `non-rigid' monoidal categories studied in connection with CFTs.
  Module categories over GV categories are discussed in \cite{2023arXiv230617668F}.
  Schweigert announced that a theory of relative Serre functors can be extended to module categories over GV categories \cite{Schweigert2023Banff}.
\end{remark}

\subsection{The dual of the internal Hom functor}
\label{subsec:dual-internal-hom}

Throughout this subsection, we fix a rigid monoidal category $\mathcal{C}$ and a closed left module category $\mathcal{M}$ over $\mathcal{C}$. We denote by $\mathcal{P}$ the full subcategory of $\mathcal{M}$ consisting of all objects $P \in \mathcal{M}$ such that the functor $\iHom(P, -)$ has a right adjoint. In the setting of Theorem \ref{thm:rel-Serre-summary}, the full subcategory $\mathcal{P}$ coincides with $\relprj{\mathcal{C}}{\mathcal{M}}$. By a similar argument as Lemma~\ref{lem:relative-proj-2}, one can prove that $\mathcal{P}$ is closed under the action of $\mathcal{C}$.

We set $S(P) = \T_P^{\rradj}(\unitobj)$ for $P \in \mathcal{P}$ and extend the assignment $P \mapsto S(P)$ to a functor from $\mathcal{P}$ to $\mathcal{M}$. The purpose of this subsection is to show that $S$ has some properties that
relative Serre functors are said to have in Theorem \ref{thm:rel-Serre-summary}.
We first prove:

\begin{lemma}
  \label{lem:rel-Serre-1}
  There is a natural isomorphism
  \begin{equation*}
    \iHom(P, M)^* \cong \iHom(M, S(P)) \quad (P \in \mathcal{P}, M \in \mathcal{M}).
  \end{equation*}
\end{lemma}
\begin{proof}
  For $P \in \mathcal{P}$ and $M \in \mathcal{M}$, there is an isomorphism
  \begin{equation}
    \label{eq:rel-Serre-sub-1-proof-1}
    \T_{\iHom(P, M)}
    = (-) \otimes \iHom(P, M)
    \xrightarrow{\quad \mathfrak{a}_{-,P,M} \quad}
    \iHom(P, (-) \catactl M)
    = \T_P^{\radj} \T_M^{}
  \end{equation}
  of left $\mathcal{C}$-module functors. By taking right adjoints, we have an isomorphism
  \begin{equation}
    \label{eq:rel-Serre-sub-1-proof-2}
    \T_M^{\radj} \T_P^{\rradj}
    \xrightarrow{\ \eqref{eq:FG-r-adj} \ } (\T_P^{\radj} \T_M^{})^{\radj}
    \xrightarrow{\ \eqref{eq:rel-Serre-sub-1-proof-1} \ }
    (\T_{\iHom(P, M)})^{\radj} = (-) \otimes \iHom(P, M)^*
  \end{equation}
  of left $\mathcal{C}$-module functors. Evaluating \eqref{eq:rel-Serre-sub-1-proof-2} at $\unitobj$, we have a natural isomorphism as stated.
\end{proof}

For $M \in \mathcal{M}$, $P \in \mathcal{P}$ and $X \in \mathcal{C}$, we denote by
\begin{equation*}
  \tilde{\phi}_{M,P}(X):
  \iHom(M, \T_P^{\rradj}(X))
  \to X \otimes \iHom(P, M)^*
\end{equation*}
the $X$-component of the isomorphism \eqref{eq:rel-Serre-sub-1-proof-2} constructed in the proof of Lemma \ref{lem:rel-Serre-1} and set $\phi_{M,P} = \tilde{\phi}_{M,P}(\unitobj)$. We define the natural isomorphism
\begin{equation*}
  \mathfrak{s}_{X,P} : X^{**} \catactl S(P) \to S(X \catactl P)
  \quad (X \in \mathcal{C}, P \in \mathcal{P})
\end{equation*}
to be the following composition:
\begin{equation*}
  \newcommand{\xarr}[1]{\xrightarrow{\makebox[4em]{$\scriptstyle #1$}}}
  \begin{aligned}
    X \catactl S(P) = X^{**} \catactl \T_P^{\rradj}(\unitobj)
    & \xarr{\omega^{(P)}_{X^{**},\unitobj}}
    \T_P^{\rradj}(X^{**})
    = \T_P^{\rradj} \T_X^{\rradj}(\unitobj) \\
    & \xarr{\eqref{eq:FG-r-adj}}
    (\T_P^{} \T_X^{})^{\rradj}(\unitobj)
    = (\T_{X \catactl P})^{\rradj}(\unitobj) = S(X \catactl P),
  \end{aligned}
\end{equation*}
where $\omega^{(P)}_{X,Y} : X \catactl \T_P^{\rradj}(Y) \to \T_P^{\rradj}(X \otimes Y)$ is the left $\mathcal{C}$-module structure of $\T_P^{\rradj}$.
By the coherence property of \eqref{eq:FG-r-adj}, one can show that $\mathfrak{s}$ makes $S$ a left $\mathcal{C}$-module functor from $\mathcal{P}$ to ${}_{(2)}\mathcal{M}$ or, equivalently, from ${}_{(-2)}\mathcal{P}$ to $\mathcal{M}$. Now we state:

\begin{lemma}
  \label{lem:rel-Serre-phi-M-P}
  If we view $S$ as a twisted left $\mathcal{C}$-module functor by $\mathfrak{s}$, then
  \begin{equation*}
    \phi_{M,P} : \iHom(M, S(P)) \to \iHom(P, M)^* \quad (M \in \mathcal{M}, P \in \mathcal{P})
  \end{equation*}
  is an isomorphism of $\mathcal{C}$-bimodule functors from $\mathcal{M}^{\op} \times {}_{(-2)}\mathcal{P}$ to $\mathcal{C}$.
\end{lemma}

We prepare some lemmas to prove Lemma~\ref{lem:rel-Serre-phi-M-P}.
For each $P \in \mathcal{P}$, we define
\begin{equation}
  \label{eq:rel-Serre-trace}
  \mathfrak{t}_{P} := \varepsilon_{P}(\unitobj) : \iHom(P, S(P)) = \T_P^{\radj} \T_P^{\rradj}(\unitobj) \to \unitobj,
\end{equation}
where $\varepsilon_P : \T_P^{\radj} \circ \T_P^{\rradj} \to \id_{\mathcal{C}}$ is the counit of the adjunction $\T_P^{\radj} \dashv \T_P^{\rradj}$.
There is the following relation between $\mathfrak{t}_P$ and $\phi_{M,P}$.

\begin{lemma}
  \label{lem:rel-Serre-trace-1}
  For $P \in \mathcal{P}$ and $M \in \mathcal{M}$, we have
  \begin{gather}
    \label{eq:rel-Serre-trace-1}
    \mathfrak{t}_{P} \circ \icomp_{P,M,S(P)}
    = \eval_{\iHom(P, M)} \circ (\phi_{P,M} \otimes \id_{\iHom(P,M)}), \\
    \label{eq:rel-Serre-trace-1'}
    \mathfrak{t}_{P} = (\icoev_{P,\unitobj})^* \circ \phi_{P,P}.
  \end{gather}
\end{lemma}
\begin{proof}
  By the construction, $\tilde{\phi}_{M,P}(X)$ is equal to the following composition:
  \begin{equation*}
    \newcommand{\xarr}[1]{\xrightarrow{\makebox[13em]{$\scriptstyle #1$}}}
    \begin{aligned}
      \T_M^{\radj} \T_P^{\rradj}(X)
      & \xarr{\text{$\id \otimes \coev_W$ with $W = \T_P^{\radj} \T_M^{}(\unitobj)$}}
      \T_M^{\radj} \T_P^{\rradj}(X) \otimes \T_P^{\radj} \T_M^{}(\unitobj) \otimes \T_P^{\radj} \T_M^{}(\unitobj)^* \\
      & \xarr{\text{$\mathfrak{a}_{W,P,M} \otimes \id$ with $W = \T_M^{\radj} \T_P^{\rradj}(X)$}}
      \T_P^{\radj} \T_M^{} (\T_M^{\radj} \T_P^{\rradj}(X)) \otimes \T_P^{\radj} \T_M^{}(\unitobj)^* \\
      & \xarr{\text{$\T_P^{\radj}(\ieval_{M, W}) \otimes \id$ with $W = \T_P^{\rradj}(X)$}}
      \T_P^{\radj} \T_P^{\rradj}(X) \otimes \T_P^{\radj} \T_M^{}(\unitobj)^* \\
      & \xarr{\varepsilon_P(X) \otimes \id}
      X \otimes \T_P^{\radj} \T_M^{}(\unitobj)^*.
    \end{aligned}
  \end{equation*}
  By \eqref{eq:iHom-composition-a-1}, we see that the composition of the second and the third arrow is equal to the morphism $\icomp_{P,M,S(P)} \otimes \id$. Thus, by letting $X = \unitobj$, we have
  \begin{align*}
    \phi_{M,P}
    = (\mathfrak{t}_{P} \otimes \id_{\iHom(M, P)^*})
    & \circ (\icomp_{P,M,S(P)} \otimes \id_{\iHom(M, P)^*}) \\
    & \circ (\id_{\iHom(M, S(P))} \otimes \coev_{\iHom(P, M)}),
  \end{align*}
  which implies equation \eqref{eq:rel-Serre-trace-1}.
  Equation \eqref{eq:rel-Serre-trace-1'} follows from \eqref{eq:rel-Serre-trace-1} with $P = M$ and the fact that $\icoev$ plays the role of unit for the composition law for the internal Hom functor.
\end{proof}

Next, we prove the following equation relating $\mathfrak{s}$ and $\mathfrak{t}$.

\begin{lemma}
  \label{lem:rel-Serre-trace-2}
  For all objects $X \in \mathcal{C}$ and $P \in \mathcal{P}$, we have
  \begin{equation}
    \label{eq:rel-Serre-trace-2}
    \mathfrak{t}_{X \catactl P}
    \circ \iHom(X \catactl P, \mathfrak{s}_{X,P}) \circ \mathfrak{c}_{X,P,S(P),X^*}
    = \eval_{X^*} \circ (\id_{X^{**}} \otimes \mathfrak{t}_{P} \otimes \id_{X^*}).
  \end{equation}
\end{lemma}
\begin{proof}
  This lemma can be proved by the same argument as in the proof of \cite[Theorem 3.9]{2019arXiv190400376S}, where $\mathcal{M}$ is assumed to be an exact left module category over a finite tensor category. For the sake of completeness, we include a proof. We consider the diagram
  \begin{equation}
    \label{eq:itrace-P-X-proof-1}
    \begin{tikzcd}[column sep = 16pt]
      \T_X^{\radj} \circ \T_P^{\radj} \circ \T_P^{\rradj} \circ \T_X^{\rradj}
      \arrow[r, "{\eqref{eq:FG-r-adj}}"]
      \arrow[d, "{\id \circ \varepsilon_P \circ \id}"']
      & (\T_P \circ \T_X)^{\radj} \circ (\T_P \circ \T_X)^{\rradj}
      \arrow[r, equal]
      & \T_{P \catactl X} \circ \T_{P \catactl X}^{\radj}
      \arrow[d, "{\varepsilon_{P \catactl X}}"] \\
      \T_X^{\radj} \circ \id \circ \T_X^{\rradj}
      \arrow[r, equal]
      & \id_{\mathcal{C}} \otimes X^{**} \otimes X^*
      \arrow[r, "{\id_{\mathcal{C}} \otimes \eval_{X^*}}"]
      & \id_{\mathcal{C}},
    \end{tikzcd}
  \end{equation}
  where the symbol $\circ$ is used to mean the composition of functors as well as the horizontal composition of natural transformations.
  We note that $\id_{\mathcal{C}} \otimes \eval_{X^*}$ in the diagram is the counit of the adjunction $\T_{X^*} \dashv \T_{X^{**}}$.
  By the construction of the isomorphism \eqref{eq:FG-r-adj}, the diagram \eqref{eq:itrace-P-X-proof-1} is commutative.

  We consider the diagram given as Figure \ref{fig:itrace-X-P}.
  The cell (1) of the diagram commutes since $\varepsilon_P : \T_P^{\radj} \T_P^{\rradj} \to \id_{\mathcal{C}}$ is a morphism of left $\mathcal{C}$-module functor.
  The cell (2) also commutes since it is actually the commutative diagram \eqref{eq:itrace-P-X-proof-1} evaluated at $\unitobj$.
  Thus the diagram of Figure \ref{fig:itrace-X-P} commutes.
  By the definition of the isomorphisms $\mathfrak{s}$ and $\mathfrak{b}$, we see that the left column of the diagram is equal to the morphism
  \begin{equation*}
    \mathfrak{b}_{P, S(X \catactl P), X} \circ (\iHom(P, \mathfrak{s}_{X,P}) \otimes \id_{X^*}) \circ (\mathfrak{a}_{X^{**}, P, S(P)} \otimes \id_{X^*}),
  \end{equation*}
  which is equal to $\iHom(P, \mathfrak{s}_{X,P}) \circ \mathfrak{c}_{X, P, S(P), X^*}$ by the naturality of $\mathfrak{b}_{P, -, X}$.
  Hence we obtain \eqref{eq:rel-Serre-trace-2}. The proof is done.
\end{proof}

\begin{figure}
  \centering
  \begin{equation*}
    \begin{tikzcd}[column sep = 100pt, row sep = 24pt]
      X^{**} \otimes \T_P^{\radj} \T_P^{\rradj}(\unitobj) \otimes X^*
      \arrow[d, "{\mathfrak{a}_{X^{**}, P, S(P)} \otimes \id_{X^*}}"']
      \arrow[r, "{\id_{X^{**}} \otimes \varepsilon_P(\unitobj) \otimes \id_{X^*}}"]
      \arrow[rdd, phantom, "{(1)}"]
      & X^{**} \otimes X^* \arrow[dd, equal] \\
      \T_P^{\radj}(X^{**} \catactl \T_P^{\rradj}(\unitobj)) \otimes X^*
      \arrow[d, "{\T_P^{\radj}(\omega^{(P)}_{X^{**}, \unitobj}) \otimes \id_{X^{*}}}"'] \\
      \T_P^{\radj} \T_P^{\rradj}(X^{**}) \otimes X^*
      \arrow[d, equal]
      \arrow[r, "{\varepsilon_P(X^{**}) \otimes \id_{X^*}}"]
      \arrow[rdd, phantom, "{(2)}"]
      & X^{**} \otimes X^* \arrow[d, "{\eval_{X^*}}"] \\
      (\T_X^{\radj} \T_P^{\radj} \T_P^{\rradj} \T_X^{\rradj})(\unitobj)
      \arrow[d, "{\eqref{eq:FG-r-adj}}"'] & \unitobj \\
      ((\T_P \T_X)^{\radj} (\T_P \T_X)^{\rradj})(\unitobj)
      \arrow[r, equal]
      & \T_{X \catactl P}^{\radj} \T_{X \catactl P}^{\rradj}(\unitobj)
      \arrow[u, "{\varepsilon_{X \catactl P}(\unitobj)}"']
    \end{tikzcd}
  \end{equation*}
  \caption{Proof of Lemma \ref{lem:rel-Serre-trace-2}}
  \label{fig:itrace-X-P}
\end{figure}

\begin{proof}[Proof of Lemma~\ref{lem:rel-Serre-phi-M-P}]
  The proof presented below is essentially same as that of \cite[Theorem 3.9]{2019arXiv190400376S} but slightly improved. We need to show that the equations
  \begin{gather}
    \label{eq:rel-Serre-sub-1-proof-4}
    \phi_{X \catactl M, P} \circ \mathfrak{b}_{M, S(P), X}
    = (\mathfrak{a}_{X, P, M}^{-1})^* \circ (\phi_{M,P} \otimes \id_{X^*}), \\
    \phi_{M, X \catactl P} \circ \iHom(M, \mathfrak{s}_{X, P}) \circ \mathfrak{a}_{X^{**}, M, S(P)}
    = (\mathfrak{b}_{P, M, X}^{-1})^* \circ (\id_{X^{**}} \otimes \phi_{M,P})
    \label{eq:rel-Serre-sub-1-proof-5}
  \end{gather}
  hold for all $X \in \mathcal{C}$, $P \in \mathcal{P}$ and $M \in \mathcal{M}$.
  Equation \eqref{eq:rel-Serre-sub-1-proof-4} is verified as follows:
  \begin{align*}
    & \eval_{X \otimes \iHom(P, M)} (((\mathfrak{a}_{X, P, M})^{*} \circ \phi_{X \catactl M, P} \circ \mathfrak{b}_{M, S(P), X})
      \otimes \id_{X \otimes \iHom(P, M)}) \\
    & = \eval_{\iHom(P, X \catactl M)} ((\phi_{X \catactl M, P} \circ \mathfrak{b}_{M, S(P), X})
      \otimes \mathfrak{a}_{X, P, M}) \\
    & = \mathfrak{t}_P \circ \icomp_{P, X \catactl M, S(P)} \circ (\mathfrak{b}_{M, S(P), X}
      \otimes \mathfrak{a}_{X, P, M}) \\
    & = \mathfrak{t}_P \circ \icomp_{P, M, S(P)} \circ (\id_{\iHom(S(P), M)} \otimes \eval_X \otimes \id_{\iHom(P, M)}) \\
    & = \eval_{X \otimes \iHom(P, M)} \circ (\phi_{M, P} \otimes \id_{X^*} \otimes \id_{X \otimes \iHom(P, M)}),
  \end{align*}
  where the first equality follows from the dinaturality of the evaluation, the second and the last from \eqref{eq:rel-Serre-trace-1}, and the third from \eqref{eq:iHom-composition-ab}.

  Equation \eqref{eq:rel-Serre-sub-1-proof-5} is verified as follows:
  \begin{align*}
    & \eval_{\iHom(M, P) \otimes X^*} \circ ((\mathfrak{b}_{P,M,X})^* \circ (\phi_{M, X \catactl P} \circ \iHom(M, \mathfrak{s}_{X, P}) \circ \mathfrak{a}_{X^{**}, M, S(P)}) \otimes \id) \\
    & = \mathfrak{t}_{X \catactl P} \circ \icomp_{X \catactl P, M, S(X \catactl P)}
      \circ ((\iHom(M, \mathfrak{s}_{X, P}) \circ \mathfrak{a}_{X^{**}, M, S(P)}) \otimes \mathfrak{b}_{P,M,X}) \\
    & = \mathfrak{t}_{X \catactl P} \circ \iHom(X \catactl P, \mathfrak{s}_{X,P})
      \circ \icomp_{X \catactl P, M, X^{**} \catactl S(P)}
      \circ (\mathfrak{a}_{X^{**}, M, S(P)} \otimes \mathfrak{b}_{P,M,X}) \\
    & = \eval_{X^*} \circ (\id_{X^{**}} \otimes \mathfrak{t}_P \otimes \id_{X^*})
      \circ (\id_{X^{**}} \otimes \icomp_{X \catactl P, M, S(P)} \otimes \id_{X^*}) \\
    & = \eval_{\iHom(M, P) \otimes X^*} \circ (\id_{X^{**}} \otimes \phi_{M, P} \otimes \id_{\iHom(M, P) \otimes X^*}),
  \end{align*}
  where the first and the last equality follow from \eqref{eq:rel-Serre-trace-1},
  the second from the naturality of $\icomp$,
  and the third from \eqref{eq:iHom-composition-a-2}, \eqref{eq:iHom-composition-b-2} and \eqref{eq:rel-Serre-trace-2}. The proof is done.
\end{proof}

\subsection{Existence and uniqueness as a functor}
\label{subsec:rel-Serre-existence}

Let $\mathcal{C}$ be a finite tensor category, and let $\mathcal{M}$ be a finite left $\mathcal{C}$-module category.
We set $S(P) = \T_{P}^{\rradj}(\unitobj)$ for $P \in \relprj{\mathcal{C}}{\mathcal{M}}$ and extend the assignment $P \mapsto S(P)$ to a functor from $\relprj{\mathcal{C}}{\mathcal{M}}$ to $\mathcal{M}$ as in Subsection \ref{subsec:dual-internal-hom}.
Lemma~\ref{lem:rel-Serre-1} says that the functor $S$ is a relative Serre functor of $\mathcal{M}$ except that its domain is only the full subcategory $\relprj{\mathcal{C}}{\mathcal{M}}$.
Since $\relprj{\mathcal{C}}{\mathcal{M}}$ contains all projective objects of $\mathcal{M}$, an extension of $S$ to a right exact endofunctor on $\mathcal{M}$ may be obtained by letting $S(M)$ for $M \in \mathcal{M}$ to be the cokernel of $S(f)$, where $f$ is a morphism in $\relprj{\mathcal{C}}{\mathcal{M}}$ whose cokernel considered in $\mathcal{M}$ is isomorphic to $M$. However, since we do not know whether $\mathcal{C}$-projective objects of $\mathcal{M}$ have a lifting property like ordinary projective objects, the well-definedness of the extension matters. To avoid this problem, we give an explicit formula of a relative Serre functor for the case where $\mathcal{M} = \mathcal{C}_A$ for some algebra $A$ in $\mathcal{C}$ as follows:

\begin{lemma}
  \label{lem:rel-Serre-mod-A}
  Let $A$ be an algebra in $\mathcal{C}$. Then the functor
  \begin{equation}
    \label{eq:rel-Serre-mod-A}
    \Ser_A := \iHom_A(-,A)^* : \mathcal{C}_A \to \mathcal{C}_A
  \end{equation}
  is a relative Serre functor of $\mathcal{C}_A$.
\end{lemma}
\begin{proof}
  It follows from the left exactness of the internal Hom functor that $\Ser_A$ is right exact.
  We fix an object $P \in \relprj{\mathcal{C}}{\mathcal{C}_A}$.
  Then, by definition, the functor $\T_P^{\radj}$ is exact. By the Eilenberg-Watts equivalence, we have a natural isomorphism
  \begin{equation*}
    \T_P^{\radj}(M) \cong M \otimes_A \T_P^{\radj}(A) = M \otimes_A \iHom_A(P, A)
    \quad (M \in \mathcal{C}_A).
  \end{equation*}
  By this expression of $\T_P^{\radj}$ and the tensor-Hom adjunction, we have
  \begin{equation*}
    \T_P^{\rradj}(X)
    \cong \iHom_{\unitobj}(\iHom_A(P, X), X)
    \cong X \otimes \iHom_A(P, A)^* = X \otimes \Ser_A(P)
  \end{equation*}
  for $X \in \mathcal{C}$.
  Hence, by Lemma~\ref{lem:rel-Serre-1}, we have natural isomorphisms
  \begin{equation*}
    \iHom_A(P, M)^* \cong \iHom_A(M, \T_P^{\rradj}(\unitobj)) \cong \iHom_A(M, \Ser_A(P))
  \end{equation*}
  for $M \in \mathcal{C}_A$ and $P \in \relprj{\mathcal{C}}{\mathcal{C}_A}$.
  Thus $\Ser_A$ is a relative Serre functor.  
\end{proof}

The following lemma should be standard in homological algebra.
We include the proof since the argument of the proof will be used frequently. 

\begin{lemma}
  \label{lem:nat-extension}
  Let $\mathcal{A}$ and $\mathcal{B}$ be abelian categories, and let $\mathcal{P}$ be a full subcategory of $\mathcal{A}$ containing all projective objects of $\mathcal{A}$. We suppose that $\mathcal{A}$ has enough projective objects. Then, for any right exact functors $F$ and $G$ from $\mathcal{A}$ to $\mathcal{B}$, the map
  \begin{equation*}
    \Nat(F, G) \to \Nat(F|_{\mathcal{P}}, G|_{\mathcal{P}}),
    \quad \xi = \{ \xi_M \}_{M \in \mathcal{M}} \mapsto \{ \xi_P \}_{P \in \mathcal{P}}
  \end{equation*}
  is a bijection.
\end{lemma}
\begin{proof}
  Given a natural transformation $\alpha : F|_{\mathcal{P}} \to G|_{\mathcal{P}}$ and $X \in \mathcal{A}$, we choose a projective resolution $\cdots \to P_1 \to P_0 \to X \to 0$ of $X$ and define $\tilde{\alpha}_X$ to be the unique morphism in $\mathcal{B}$ such that the diagram
  \begin{equation*}
    \begin{tikzcd}
      F(P_1)   \arrow[r] \arrow[d, "{\alpha_{P_1}}"]
      & F(P_0) \arrow[r] \arrow[d, "{\alpha_{P_0}}"]
      & F(X)   \arrow[r] \arrow[d, "{\tilde{\alpha}_X}"] & 0 \\
      G(P_1) \arrow[r] & G(P_0) \arrow[r] & G(X) \arrow[r] & 0
    \end{tikzcd}
  \end{equation*}
  is commutative. By the uniqueness of projective resolutions up to chain homotopy, one can show that the morphism $\tilde{\alpha}_X$ does not depend on the choice of the projective resolution of $X$. The family $\tilde{\alpha} := \{ \tilde{\alpha}_X \}_{X \in \mathcal{A}}$ defines a natural transformation from $F$ to $G$.
  By the construction, $\tilde{\alpha}_P = \alpha_P$ if $P$ is projective.
  For an object $Q \in \mathcal{P}$ which is not necessarily projective, we take an epimorphism $\pi : P \to Q$ in $\mathcal{A}$ with $P$ projective. By the naturality of $\alpha$ and $\tilde{\alpha}$, we have
  \begin{equation*}
    \alpha_Q \circ F(\pi) = F(\pi) \circ \alpha_P
    = F(\pi) \circ \tilde{\alpha}_Q = \tilde{\alpha}_P \circ F(\pi)
  \end{equation*}
  and hence $\alpha_Q = \tilde{\alpha}_Q$.
  Now it is easy to check that the assignment $\alpha \mapsto \tilde{\alpha}$ gives the inverse of the map in the statement.
\end{proof}

\begin{proof}[Proof of the existence and uniqueness part of Theorem~\ref{thm:rel-Serre-summary}]
  Let $\mathcal{M}$ be a finite left $\mathcal{C}$-module category.
  It is obvious that the existence of a relative Serre functor is an invariant of a finite left $\mathcal{C}$-module category.
  Since $\mathcal{M} \approx \mathcal{C}_A$ for some algebra $A$ in $\mathcal{C}$, the existence of a relative Serre functor follows from Lemma~\ref{lem:rel-Serre-mod-A}.

  Now let $\Ser$ and $\Ser'$ be relative Serre functors on $\mathcal{M}$. Then we have natural isomorphisms $\iHom(M, \Ser(P)) \cong \iHom(P, M)^* \cong \iHom(M, \Ser'(P))$ for $M \in \mathcal{M}$ and $P \in \relprj{\mathcal{C}}{\mathcal{M}}$.
  Hence, by the internal Yoneda Lemma~\ref{lem:internal-Yoneda-0}, we see that $\Ser$ and $\Ser'$ become isomorphic when they are restricted to $\relprj{\mathcal{C}}{\mathcal{M}}$. Lemma \ref{lem:nat-extension} concludes that $\Ser$ and $\Ser'$ are isomorphic on the whole category $\mathcal{M}$. The proof is done.
\end{proof}

\subsection{Twisted module structure of $\Ser$}
\label{subsec:rel-Serre-twisted-module}

Here we prove Theorem \ref{thm:rel-Serre-summary} (\ref{item:rel-Serre-main-1}).
Let $\mathcal{C}$ be a finite tensor category, and let $\mathcal{M}$ be a finite left $\mathcal{C}$-module category with relative Serre functor $\Ser$.
By the uniqueness of a relative Serre functor as a functor, we may assume that the restriction of $\Ser$ to $\relprj{\mathcal{C}}{\mathcal{M}}$ is given by $\Ser(P) = \T_P^{\rradj}(\unitobj)$.
Thus, by Lemma \ref{lem:rel-Serre-phi-M-P}, the restriction of $\Ser$ to $\relprj{\mathcal{C}}{\mathcal{M}}$ has a twisted $\mathcal{C}$-module structure as desired. Thus, to prove Theorem \ref{thm:rel-Serre-summary} (\ref{item:rel-Serre-main-1}), it suffices to show that the twisted module structure of the restriction of $\Ser$ can be extended to $\Ser$. We discuss this problem in a general form and prove the following lemma:

\begin{lemma}
  \label{lem:extension-module-structure}
  Let $\mathcal{M}$ and $\mathcal{N}$ be finite left $\mathcal{C}$-module categories, let $\mathcal{P}$ be a full subcategory of $\mathcal{M}$ closed under the action of $\mathcal{C}$ and containing all projective objects of $\mathcal{M}$, and let $F: \mathcal{M} \to \mathcal{N}$ be a right exact functor. Suppose that there is a natural isomorphism
  \begin{equation*}
    \omega_{X,P} : X \catactl F(P) \to F(X \catactl P)
    \quad (X \in \mathcal{C}, P \in \mathcal{P})
  \end{equation*}
  making the restriction of $F$ to $\mathcal{P}$ a left $\mathcal{C}$-module functor from $\mathcal{P}$ to $\mathcal{M}$. Then there is a unique natural isomorphism
  \begin{equation*}
    \tilde{\omega}_{X,M} : X \catactl F(M) \to F(X \catactl M)
    \quad (X \in \mathcal{C}, M \in \mathcal{M})
  \end{equation*}
  that extends $\omega$. Moreover, the natural transformation $\tilde{\omega}$ makes $F$ a left $\mathcal{C}$-module functor from $\mathcal{M}$ to $\mathcal{N}$.
\end{lemma}
\begin{proof}
  For $X \in \mathcal{C}$ and $M \in \mathcal{M}$, we define $\tilde{\omega}_{X,M} : X \catactl F(M) \to F(X \catactl M)$ to be the $M$-component of the natural transformation obtained by extending the natural transformation $\{ \omega_{X,P} \}_{P \in \mathcal{P}}$ by Lemma \ref{lem:nat-extension}.
  It is obvious that $\tilde{\omega}_{X,M}$ is natural in $M$.
  The functorial property of the bijection given in Lemma \ref{lem:nat-extension} implies that $\tilde{\omega}_{X,M}$ is also natural in the variable $X$.
  The uniqueness of such a natural transformation $\tilde{\omega}$ also follows from Lemma \ref{lem:nat-extension}.

  It remains to show that $\tilde{\omega}$ makes $F$ a left $\mathcal{C}$-module functor. Let $X, Y \in \mathcal{C}$ and $M \in \mathcal{M}$ be objects. We choose an epimorphism $\pi : P \to M$ in $\mathcal{M}$ from a projective object $P \in \mathcal{M}$. Then we compute
  \begin{align*}
    \tilde{\omega}_{X \otimes Y, M} \circ (\id_X \catactl \id_Y \catactl F(\pi))
    & = F(\id_X \catactl \id_Y \catactl \pi) \circ \omega_{X \otimes Y, P} \\
    & = F(\id_X \catactl \id_Y \catactl \pi) \circ \omega_{X, Y \catactl P} \circ (\id_X \catactl \omega_{Y, P}) \\
    & = \tilde{\omega}_{X, Y \catactl M} \circ (\id_X \catactl \tilde{\omega}_{Y, M})
      \circ (\id_X \catactl \id_Y \catactl F(\pi)),
  \end{align*}
  where the second equality follows from the assumption that the restriction of $F$ to $\mathcal{P}$ is a left $\mathcal{C}$-module functor with structure morphism $\omega$. Since $\id_X \catactl \id_Y \catactl F(\pi)$ is epic, we have $\tilde{\omega}_{X \otimes Y, M} = \tilde{\omega}_{X, Y \catactl M} \circ (\id_X \catactl \tilde{\omega}_{Y, M})$.
  We also have $\tilde{\omega}_{\unitobj, M} \circ F(\pi) = F(\pi) \circ \omega_{\unitobj, P} = F(\pi)$,
  which implies $\tilde{\omega}_{\unitobj, M} = \id_{F(M)}$. The proof is done.
\end{proof}

\begin{proof}[Proof of Theorem \ref{thm:rel-Serre-summary} (\ref{item:rel-Serre-main-1})]
  Lemma~\ref{lem:extension-module-structure} completes the discussion at the beginning of this subsection.
\end{proof}

\subsection{Uniqueness as a twisted module functor}
\label{subsec:rel-Serre-uniqueness}

As we have proved, a relative Serre functor is unique up to isomorphism as a functor.
Here we show that a relative Serre functor is unique up to canonical isomorphism preserving accompanying structure morphisms. To explain what this means, we introduce the following terminology:

\begin{definition}
  \label{def:rel-Serre-structured}
  A {\em structured relative Serre functor} of $\mathcal{M}$ is a triple $(\Ser, \mathfrak{s}, \phi)$ consisting of a right exact functor $\Ser: \mathcal{M} \to \mathcal{M}$, a natural isomorphism
  \begin{equation*}
    \mathfrak{s}_{X,M} : X^{**} \catactl \Ser(M) \to \Ser(X \catactl M)
    \quad (X \in \mathcal{C}, M \in \mathcal{M})
  \end{equation*}
  making $\Ser$ a twisted left $\mathcal{C}$-module functor, and a natural isomorphism
  \begin{equation*}
    \phi_{M,P} : \iHom(M, \Ser(P)) \to \iHom(P, M)^*
    \quad (M \in \mathcal{M}, P \in \relprj{\mathcal{C}}{\mathcal{M}})
  \end{equation*}
  of $\mathcal{C}$-bimodule functors from $\mathcal{M}^{\op} \times {}_{(-2)}\relprj{\mathcal{C}}{\mathcal{M}}$ to $\mathcal{C}$.
  Given a structured relative Serre functor $(\Ser, \mathfrak{s}, \phi)$, we define the {\em trace} by
  \begin{equation*}
    \itrace_P := (\icoev_{P, \unitobj})^* \circ \phi_{P,P} : \iHom(P, S(P)) \to \unitobj \quad (P \in \relprj{\mathcal{C}}{\mathcal{M}}).
  \end{equation*}
\end{definition}

Theorem \ref{thm:rel-Serre-summary} (\ref{item:rel-Serre-main-1}) says that a relative Serre functor of $\mathcal{M}$ can be extended to a structured relative Serre functor of $\mathcal{M}$. We now formulate and prove the uniqueness of a structured relative Serre functor as follows:

\begin{proposition}
  \label{prop:rel-Serre-uniq-structured}
  Let $(\Ser, \mathfrak{s}, \phi)$ and $(\Ser', \mathfrak{s}', \phi')$ be structured relative Serre functors of $\mathcal{M}$, and let $\itrace$ and $\itrace'$ be associated traces.
  Then there is a unique isomorphism $\theta : \Ser \to \Ser'$ of functors such that the equation
  \begin{equation}
    \label{eq:rel-Serre-uniq-structured-1}
    \phi'_{M,P} \circ \iHom(M, \theta_P) = \phi_{M,P}
  \end{equation}
  holds for all objects $M \in \mathcal{M}$ and $P \in \relprj{\mathcal{C}}{\mathcal{M}}$. The isomorphism $\theta$ satisfies
  \begin{gather}
    \label{eq:rel-Serre-uniq-structured-2}
    \mathfrak{s}'_{X,M} \circ (\id_X \catactl \theta_M) = \theta_{X \catactl M} \circ \mathfrak{s}_{X,M}, \\
    \label{eq:rel-Serre-uniq-structured-3}
    \itrace'_P = \itrace_P \circ \iHom(M, \theta_P)
  \end{gather}
  for all objects $X \in \mathcal{C}$, $M \in \mathcal{M}$ and $P \in \relprj{\mathcal{C}}{\mathcal{M}}$.
\end{proposition}

This proposition implies that there is an isomorphism of twisted left $\mathcal{C}$-module functors between $(\Ser, \mathfrak{s})$ and $(\Ser', \mathfrak{s}')$. However, such an isomorphism is not unique (there are, at least, non-zero scalar multiples of the isomorphism $\theta$ of this proposition). This proposition says that an isomorphism between $(\Ser, \mathfrak{s})$ and $(\Ser', \mathfrak{s}')$ that is compatible with $\phi$ and $\phi'$ is unique.

\begin{proof}
  By Lemma~\ref{lem:internal-Yoneda-1}, there is a unique natural isomorphism $\theta_P : \Ser(P)\to \Ser'(P)$ ($P \in \relprj{\mathcal{C}}{\mathcal{M}}$) such that the equation
  \begin{equation*}
    \iHom(M, \theta_P) = \phi_{M,P} \circ (\phi'_{M,P})^{-1} : \iHom(M, \Ser(P)) \to \iHom(M, \Ser'(P))
  \end{equation*}
  holds for all objects $M \in \mathcal{M}$. We extend $\theta$ to an isomorphism from $\Ser$ to $\Ser'$ by Lemma~\ref{lem:nat-extension} and obtain an isomorphism $\theta : \Ser \to \Ser'$ satisfying \eqref{eq:rel-Serre-uniq-structured-1}. The uniqueness of such an isomorphism follows from Lemmas~\ref{lem:internal-Yoneda-1} and~\ref{lem:nat-extension}.

  The internal Yoneda Lemma~\ref{lem:internal-Yoneda-2} implies that  $\theta$ is in fact an isomorphism of $\mathcal{C}$-module functors between restrictions of $(\Ser, \mathfrak{s})$ and $(\Ser', \mathfrak{s}')$ to $\relprj{\mathcal{C}}{\mathcal{M}}$. Namely, equation \eqref{eq:rel-Serre-uniq-structured-2} holds for all objects $X \in \mathcal{C}$ and $M \in \relprj{\mathcal{C}}{\mathcal{M}}$.
  By Lemma~\ref{lem:nat-extension}, we conclude that \eqref{eq:rel-Serre-uniq-structured-2} actually holds for all $X \in \mathcal{C}$ and $M \in \mathcal{M}$. Equation \eqref{eq:rel-Serre-uniq-structured-3} is obvious from \eqref{eq:rel-Serre-uniq-structured-1} and the definition of the trace.
\end{proof}

As an application of Proposition \ref{prop:rel-Serre-uniq-structured}, we prove the following equations involving the trace, which will be used to construct Frobenius algebras in a finite tensor category in later sections.

\begin{proposition}
  Let $(\Ser, \mathfrak{s}, \phi)$ be a structured relative Serre functor of $\mathcal{M}$, and let $\itrace$ be the associated trace. Then we have
  \begin{gather}
    \label{eq:itrace-1}
    \itrace_{P} \circ \icomp_{P,M,S(P)}
    = \eval_{\iHom(P, M)} \circ (\phi_{P,M} \otimes \id_{\iHom(P,M)}), \\
    \label{eq:itrace-2}
    \itrace_{X \catactl P}
    \circ \iHom(X \catactl P, \mathfrak{s}_{X,P}) \circ \mathfrak{c}_{X,P,S(P),X^*}
    = \eval_{X^*} \circ (\id_{X^{**}} \otimes \itrace_{P} \otimes \id_{X^*})
  \end{gather}
  for all objects $X \in \mathcal{C}$, $M \in \mathcal{M}$ and $P \in \relprj{\mathcal{C}}{\mathcal{M}}$.
\end{proposition}
\begin{proof}
  Lemmas \ref{lem:rel-Serre-trace-1} and \ref{lem:rel-Serre-trace-2} show that there is, at least, one structured relative Serre functor fulfilling equations \eqref{eq:itrace-1} and~\eqref{eq:itrace-2}. By Proposition \ref{prop:rel-Serre-uniq-structured}, all structured relative Serre functors satisfy \eqref{eq:itrace-1} and~\eqref{eq:itrace-2}.
\end{proof}

\subsection{Relation to the Nakayama functor}
\label{subsec:rel-Serre-Nakayama}

Let $\mathcal{C}$ be a finite tensor category, and let $\mathcal{M}$ be a finite left $\mathcal{C}$-module category with relative Serre functor $\Ser$.
We define $\alpha = \Nak_{\mathcal{C}}(\unitobj)$.
We prove Theorem \ref{thm:rel-Serre-summary} (\ref{item:rel-Serre-main-2}), which states that $\alpha \catactl \Ser$ is isomorphic to $\Nak_{\mathcal{M}}$ as twisted left $\mathcal{C}$-module functors.
Here, $\alpha \catactl \Ser$ is viewed as a twisted left $\mathcal{C}$-module functor by the structure morphism given by
\begin{equation*}
  {}^{**}\!X \catactl \alpha \catactl \Ser(M)
  \xrightarrow{\quad \mathfrak{r}_X \catactl \id_{\Ser(M)} \quad}
  \alpha \catactl X^{**} \catactl \Ser(M)
  \xrightarrow{\quad \id_{\alpha} \catactl \mathfrak{s}_{X,M} \quad}
  \alpha \catactl \Ser(X \catactl M)
\end{equation*}
for $X \in \mathcal{C}$ and $M \in \mathcal{M}$, where $\mathfrak{r}$ is the Radford isomorphism and $\mathfrak{s}$ is the twisted left $\mathcal{C}$-module structure of $\Ser$ given in Theorem \ref{thm:rel-Serre-summary} (\ref{item:rel-Serre-main-1}).

\begin{proof}[Proof of Theorem \ref{thm:rel-Serre-summary} (\ref{item:rel-Serre-main-2})]
  By Proposition \ref{prop:rel-Serre-uniq-structured}, we may assume that $\Ser$ is given by $\Ser(P) = \T_P^{\rradj}(\unitobj)$ for $P \in \relprj{\mathcal{C}}{\mathcal{M}}$ and the twisted $\mathcal{C}$-module structure of $\Ser$ is same as that discussed in Lemma~\ref{lem:rel-Serre-1}.  
  For $P \in \relprj{\mathcal{C}}{\mathcal{M}}$, we define
  \begin{equation}
    \label{eq:rel-Serre-main-thm-proof-part-3}
    \newcommand{\xarr}[1]{\xrightarrow{\makebox[5em]{$\scriptstyle #1$}}}
    \begin{aligned}
      \theta_P = \Big(\alpha \catactl \Ser(P)
      = \alpha \catactl \T_P^{\rradj}(\unitobj)
      & \xarr{\omega_{\alpha,\unitobj}^{(P)}} \T_P^{\rradj}(\alpha)
      = \T_P^{\rradj} \Nak_{\mathcal{C}}(\unitobj) \\
      & \xarr{\mathfrak{n}_{\T_P}(\unitobj)}
      \Nak_{\mathcal{M}}\T_P(\unitobj)
      = \Nak_{\mathcal{M}}(P)\Big),
    \end{aligned}
  \end{equation}
  where $\omega^{(P)}$ is the left $\mathcal{C}$-module structure of $\T_P^{\rradj}$ and $\mathfrak{n}$ is the isomorphism \eqref{eq:Nakayama-cano-iso} coming with the Nakayama functor.

  Now we consider the diagram given as Figure~\ref{fig:verify-theta}.
  The cell (1) of the diagram commutes by the naturality of $\omega^{(P)}$, and the neighboring triangles also commute since $\omega^{(P)}$ is the structure morphism of a left $\mathcal{C}$-module functor.
  The cell (2) commutes since $\mathfrak{n}_F$ is in fact an isomorphism of twisted $\mathcal{C}$-module functors when $F$ is a $\mathcal{C}$-module functor.
  The cell (3) commutes since the canonical isomorphism \eqref{eq:FG-r-adj} is in fact an isomorphism of $\mathcal{C}$-module functors when $F$ and $G$ are $\mathcal{C}$-module functors.
  Finally, the cell (4) commutes by the coherence property of $\mathfrak{n}_F$.
  Hence the diagram of Figure~\ref{fig:verify-theta} commutes.

  The left and the right column of the diagram are the twisted left $\mathcal{C}$-module structure of $\alpha \catactl \Ser(-)$ and $\Nak_{\mathcal{M}}$, respectively.
  The top and the bottom row of the diagram are $\id_{{}^{**}X} \catactl \theta_P$ and $\theta_{X \catactl P}$, respectively.
  Hence we have proved that $\alpha \catactl \Ser$ and $\Nak_{\mathcal{M}}$ are isomorphic as twisted left $\mathcal{C}$-module functors when they are restricted to $\relprj{\mathcal{C}}{\mathcal{M}}$. Now we use Lemma~\ref{lem:nat-extension} to conclude that the functors $\alpha \catactl \Ser$ and $\Nak_{\mathcal{M}}$ are in fact isomorphic as left $\mathcal{C}$-module functors from $\mathcal{M}$ to ${}_{(-2)}\mathcal{M}$. The proof is done.
\end{proof}

\begin{figure}
  \begin{equation*}
    \begin{tikzcd}[column sep = 48pt, row sep = 24pt]
      {}^{**}X \catactl \alpha \catactl \T_P^{\rradj}(\unitobj)
      \arrow[r, "{\id \catactl \omega^{(P)}}"]
      \arrow[d, "{\mathfrak{r}_X \catactl \id}"']
      \arrow[rd, "{\omega^{(P)}}"]
      \arrow[rdd, phantom, "{(1)}"]
      & {}^{**}X \catactl \T_P^{\rradj} \Nak_{\mathcal{C}}(\unitobj)
      \arrow[r, "{\id \catactl \mathfrak{n}_{\T_P}(\unitobj)}"]
      \arrow[d, "{\omega^{(P)}}"]
      \arrow[rdd, phantom, "{(2)}"]
      & {}^{**}X \catactl \Nak_{\mathcal{M}} \T_P(\unitobj)
      \arrow[d, "{\mathfrak{n}_{X \catactl \id}}"] \\
      \alpha \catactl X^{**} \catactl \T_P^{\rradj}(\unitobj)
      \arrow[d, "{\id \catactl \omega^{(P)}}"']
      \arrow[rd, "{\omega^{(P)}}"]
      & \T_P^{\rradj}({}^{**}X \otimes \Nak_{\mathcal{C}}(\unitobj))
      \arrow[d, "{\T_P^{\rradj}(\mathfrak{r}_X)}"]
      & \Nak_{\mathcal{M}}(X \catactl \T_P(\unitobj))
      \arrow[d, equal] \\
      \alpha \catactl \T_P^{\rradj}(X^{**})
      \arrow[r, "{\omega^{(P)}}"]
      \arrow[d, equal]
      \arrow[rdd, phantom, "{(3)}"]
      & \T_P^{\rradj} \Nak_{\mathcal{C}}(X)
      \arrow[r, "{\mathfrak{n}_{\T_P}}"]
      \arrow[d, equal]
      \arrow[rdd, phantom, "{(4)}"]
      & \Nak_{\mathcal{M}} \T_P(X)
      \arrow[dd, equal] \\
      \alpha \catactl \T_P^{\rradj} \T_X^{\rradj}(\unitobj)
      \arrow[d, "{\eqref{eq:FG-r-adj}}"']
      & \T_P^{\rradj} \T_X^{\rradj}(\alpha)
      \arrow[d, "{\eqref{eq:FG-r-adj}}"] \\
      \alpha \catactl \T_{X \catactl P}^{\rradj}
      \arrow[r, "{\omega^{(X \catactl P)}}"]
      & \T_{X \catactl P}^{\rradj} \Nak_{\mathcal{C}}(\unitobj)
      \arrow[r, "{\mathfrak{n}_{\T_{X \catactl P}}}"]
      & \Nak_{\mathcal{M}} \T_{X \catactl P}(\unitobj)
    \end{tikzcd}
  \end{equation*}
  \caption{Verification that $\theta$ is a morphism of $\mathcal{C}$-module functors}
  \label{fig:verify-theta}
\end{figure}

\subsection{Equivalence between $\relprj{\mathcal{C}}{\mathcal{M}}$ and $\relinj{\mathcal{C}}{\mathcal{M}}$}
\label{subsec:rel-Serre-equiv-PCM-ICM}

Let $\mathcal{C}$ be a finite tensor category, and let $\mathcal{M}$ be a finite left $\mathcal{C}$-module category.
Here we prove Theorem \ref{thm:rel-Serre-summary} (\ref{item:rel-Serre-main-3}), which states that a relative Serre functor of $\mathcal{M}$ induces an equivalence from $\relprj{\mathcal{C}}{\mathcal{M}}$ and $\relinj{\mathcal{C}}{\mathcal{M}}$.
A key observation is:

\begin{lemma}
  \label{lem:rel-Serre-ra-iHom}
  There is a natural isomorphism
  \begin{equation*}
    \Hom(\overline{\Ser}(E), M) \cong {}^*\iHom(M, E)
    \quad (M \in \mathcal{M}, E \in \relinj{\mathcal{C}}{\mathcal{M}}),
  \end{equation*}
  where $\overline{\Ser}$ is a right adjoint of a relative Serre functor of $\mathcal{M}$.
\end{lemma}

To prove this lemma, it suffices to consider the case where $\mathcal{M} = \mathcal{C}_A$ for some algebra $A$ in $\mathcal{C}$.
For an algebra $A$ in $\mathcal{C}$, we have already constructed a relative Serre functor $\Ser_A : \mathcal{C}_A \to \mathcal{C}_A$ in Lemma \ref{lem:rel-Serre-mod-A}.

\begin{lemma}
  \label{lem:rel-Serre-mod-A-ra}
  The functor $\Ser_A$ has a right adjoint
  \begin{equation}
    \overline{\Ser}_A = {}^{**}\iHom_A(A^*, -) : \mathcal{C}_A \to \mathcal{C}_A.
  \end{equation}
\end{lemma}
\begin{proof}
  Since $\Ser_A$ is a left $\mathcal{C}$-module functor from ${}_{(-2)}\mathcal{C}_A$ to $\mathcal{C}_A$, the composition
  \begin{equation*}
    F : \mathcal{C}_{A^{**}}
    \xrightarrow{\quad {}^{**}(-) \quad}
    {}_{(-2)}\mathcal{C}_A
    \xrightarrow{\quad \Ser_A \quad}
    \mathcal{C}_A
  \end{equation*}
  is a left $\mathcal{C}$-module functor. By the Eilenberg-Watts equivalence, we have
  \begin{equation*}
    F(M)
    \cong M \otimes_{A^{**}} F(A^{**})
    \cong M \otimes_{A^{**}} \iHom(A, A)^*
    \cong M \otimes_{A^{**}} A^*
  \end{equation*}
  for $M \in \mathcal{C}_{A^{**}}$. Thus we have a natural isomorphism
  \begin{equation*}
    \Ser_A(M) = F(M^{**}) \cong M^{**} \otimes_{A^{**}} A^*
  \end{equation*}
  for $M \in \mathcal{C}_A$.
  By the tensor-Hom adjunction, $\overline{\Ser}_A$ is right adjoint to $\Ser_A$, as stated.
\end{proof}

\begin{proof}[Proof of Lemma~\ref{lem:rel-Serre-ra-iHom}]
  We may, and do, assume that $\mathcal{M} = \mathcal{C}_A$ for some algebra $A$ in $\mathcal{C}$.
  We fix an object $E \in \relinj{\mathcal{C}}{\mathcal{C}_A}$. By Lemma \ref{lem:C-representability-2}, the functor
  \begin{equation*}
    {}^*\iHom_A(-,E) : \mathcal{M} \to \mathcal{C}
  \end{equation*}
  is $\mathcal{C}$-representable. Let $\overline{S}(E) \in \mathcal{C}_A$ be an object $\mathcal{C}$-representing this functor. The assignment $E \mapsto \overline{S}(E)$ extends to a functor from $\relinj{\mathcal{C}}{\mathcal{M}}$ to $\mathcal{M}$ and, by definition, there is a natural isomorphism
  \begin{equation}
    \label{eq:proof-rel-Serre-ra-iHom-1}
    {}^*\iHom_A(M, E) \cong \iHom_A(\overline{S}(E), M)
    \quad (M \in \mathcal{C}_A, E \in \relinj{\mathcal{C}}{\mathcal{C}_A}).
  \end{equation}
  By putting $M = A^*$ and using Lemma~\ref{lem:iHom-dual}, we have
  \begin{equation*}
    \overline{S}(E)
    \cong {}^*\iHom_A(\overline{S}(E), A^*)
    \cong {}^{**}\iHom_A(A^*, E) = \overline{\Ser}_A(E).
  \end{equation*}
  This means that the functor $\overline{S}$ is isomorphic to the restriction of $\overline{\Ser}_A$.
  Now the claim follows from the natural isomorphism \eqref{eq:proof-rel-Serre-ra-iHom-1}.
\end{proof}

\begin{proof}[Proof of Theorem \ref{thm:rel-Serre-summary} (\ref{item:rel-Serre-main-3})]
  Let $\Ser$ and $\overline{\Ser}$ be a relative Serre functor of $\mathcal{M}$ and its right adjoint, respectively.
  There are natural isomorphisms
  \begin{equation*}
    \iHom(P, M)^* \cong \iHom(M, \Ser(P)),
    \quad
    {}^*\iHom(M, E) \cong \iHom(\overline{\Ser}(E), M)
  \end{equation*}
  for $M \in \mathcal{M}$, $P \in \relprj{\mathcal{C}}{\mathcal{M}}$ and $E \in \relinj{\mathcal{C}}{\mathcal{M}}$.
  Since $\iHom(P, -)$ is exact, $\Ser(P)$ needs to be $\mathcal{C}$-injective.
  Similarly, $\overline{\Ser}(E)$ is $\mathcal{C}$-projective.
  Now we denote by
  \begin{equation*}
    \Ser_{*}: \relprj{\mathcal{C}}{\mathcal{M}} \to \relinj{\mathcal{C}}{\mathcal{M}} \quad \text{and} \quad
    \overline{\Ser}_{*}: \relprj{\mathcal{C}}{\mathcal{M}} \to \relinj{\mathcal{C}}{\mathcal{M}}
  \end{equation*}
  the functors induced by $\Ser$ and $\overline{\Ser}$, respectively.
  For $M \in \mathcal{M}$ and $P \in \relprj{\mathcal{C}}{\mathcal{M}}$, we have natural isomorphisms
  \begin{align*}
    \iHom(\overline{\Ser} \Ser(P), M)
    \cong {}^*\iHom(M, \Ser(P))
    \cong {}^*\iHom(P, M)^* = \iHom(P, M),
  \end{align*}
  which implies $\overline{\Ser}_{*} \circ \Ser_{*} \cong \id$ by the internal Yoneda Lemma~\ref{lem:internal-Yoneda-0}. In a similar manner, we also have $\Ser_{*} \circ \overline{\Ser}_{*} \cong \id$. Therefore the functor $\Ser_{*}$ is an equivalence with quasi-inverse $\overline{\Ser}_{*}$. The proof is done.
\end{proof}

\section{Quasi-Frobenius algebras}
\label{sec:qf-algebras}

\subsection{Lemmas on $\mathcal{C}$-projective/injective objects}
\label{subsec:C-proj-C-inj}

Throughout this section, $\mathcal{C}$ is a finite tensor category.
We recall that a ring $R$ is said to be quasi-Frobenius if it is Noetherian and injective as a right $R$-module \cite[Chapter 6]{MR1653294}.
Due to the finiteness of $\mathcal{C}$, every algebra in $\mathcal{C}$ is Noetherian.
We therefore propose the following definition:

\begin{definition}
  \label{def:QF-algebras}
  A {\em quasi-Frobenius algebra} in $\mathcal{C}$ is an algebra $A$ in $\mathcal{C}$ such that the right $A$-module $A$ is $\mathcal{C}$-injective.
\end{definition}

The aim of this section is to give equivalent conditions for an algebra $A$ in $\mathcal{C}$ to be quasi-Frobenius (Theorem \ref{thm:QF-algebras}).
For this purpose, in this subsection, we provide some lemmas on $\mathcal{C}$-projective and $\mathcal{C}$-injective objects in a finite left $\mathcal{C}$-module category.
As a preparation, we note:

\begin{lemma}
  \label{lem:ftc-progenerator}
  For any non-zero object $X \in \mathcal{C}$, there is a projective object $P \in \mathcal{C}$ such that $P \otimes X$ is a projective generator of $\mathcal{C}$.
\end{lemma}
\begin{proof}
  Let $\{ V_1, \cdots, V_n \}$ be a set of complete representatives of the isomorphism classes of simple objects of $\mathcal{C}$, and let $P_{\unitobj}$ be the projective cover of the unit object of $\mathcal{C}$ with epimorphism $\pi : P_{\unitobj} \to \unitobj$. We set $P = \bigoplus_{i = 1}^n P_{\unitobj} \otimes V_i \otimes X^*$. Then $P$ is projective and, for each $i$, there is an epimorphism
  \begin{equation*}
    P \otimes X
    \xrightarrow{\quad \text{projection} \quad} P_{\unitobj} \otimes V_i \otimes X^* \otimes X
    \xrightarrow{\quad \pi \otimes \id \otimes \eval_X \quad} V_i.
  \end{equation*}
  Thus $P \otimes X$ is a projective generator. The proof is done.
\end{proof}

Let $\mathcal{M}$ be a finite left $\mathcal{C}$-module category.
In Lemmas~\ref{lem:C-proj} and \ref{lem:C-proj-direct-summands} below, we give conditions for an object of $\mathcal{M}$ to be $\mathcal{C}$-projective.
Lemma~\ref{lem:C-proj} also says that an exact left $\mathcal{C}$-module category is precisely a finite left $\mathcal{C}$-module category of which every object is $\mathcal{C}$-projective. 

\begin{lemma}
  \label{lem:C-proj}
  For an object $M \in \mathcal{M}$, the following are equivalent:
  \begin{enumerate}
  \item $M$ is $\mathcal{C}$-projective.
  \item $P \catactl M$ is projective for all projective $P \in \mathcal{C}$.
  \item There is a non-zero object $X \in \mathcal{C}$ such that $X \catactl M$ is $\mathcal{C}$-projective.
  \end{enumerate}
\end{lemma}
\begin{proof}
  (1) $\Rightarrow$ (2). Suppose that $M$ is $\mathcal{C}$-projective. There is an isomorphism
  \begin{equation*}
    \Hom_{\mathcal{M}}(P \catactl M, -) \cong \Hom_{\mathcal{M}}(P, \iHom_{\mathcal{M}}(M, -))
  \end{equation*}
  of functors for $P \in \mathcal{C}$. If $P$ is projective, then the right-hand side is exact. This means that $P \catactl M$ is projective.

  \smallskip
  (2) $\Rightarrow$ (3). Let $X$ be a non-zero projective object of $\mathcal{C}$. If (2) holds, then $X \catactl M$ is projective, and hence it is $\mathcal{C}$-projective by Lemma~\ref{lem:relative-proj-1}. Thus (3) holds.

  \smallskip
  (3) $\Rightarrow$ (1). Suppose that (3) holds and take a non-zero object $X \in \mathcal{C}$ such that $X \catactl M$ is $\mathcal{C}$-projective.
  By Lemma~\ref{lem:ftc-progenerator}, there is an object $P \in \mathcal{C}$ such that $G := P \otimes X$ is a projective generator of $\mathcal{C}$. There is an isomorphism
  \begin{equation*}
    \Hom_{\mathcal{M}}(G, \iHom(M, -))
    \cong \Hom_{\mathcal{M}}(P, \iHom(X \catactl M, -))
  \end{equation*}
  of functors. Since $X \catactl M$ is $\mathcal{C}$-projective, the right-hand side is exact. Thus the left-hand side is also exact. Since $\Hom_{\mathcal{M}}(G, -)$ reflects exact sequences, we conclude that $\iHom(M, -)$ is exact. The proof is done.
\end{proof}

\begin{lemma}
  \label{lem:C-proj-direct-summands}
  Let $X$ and $Y$ be objects of $\mathcal{M}$.
  Then $X \oplus Y$ is $\mathcal{C}$-projective if and only if both $X$ and $Y$ are.
\end{lemma}
\begin{proof}
  This follows from Lemma~\ref{lem:C-proj} and the fact that a direct summand of a projective object in an abelian category is again projective.
\end{proof}

We give conditions for an object of $\mathcal{M}$ to be $\mathcal{C}$-injective.
Lemmas \ref{lem:C-inj} and \ref{lem:C-inj-direct-summands} below are obtained by applying Lemmas \ref{lem:C-proj} and \ref{lem:C-proj-direct-summands}, respectively, to the finite left $\mathcal{C}^{\op}$-module category $\mathcal{M}^{\op}$ in view of \eqref{eq:relative-proj-opposite}.

\begin{lemma}
  \label{lem:C-inj}
  For an object $M \in \mathcal{M}$, the following are equivalent:
  \begin{enumerate}
  \item $M$ is $\mathcal{C}$-injective.
  \item $E \catactl M$ is injective for all injective $E \in \mathcal{C}$.
  \item There is a non-zero object $X \in \mathcal{C}$ such that $X \catactl M$ is $\mathcal{C}$-injective.
  \end{enumerate}
\end{lemma}



\begin{lemma}
  \label{lem:C-inj-direct-summands}
  Let $X$ and $Y$ be objects of $\mathcal{M}$.
  Then $X \oplus Y$ is $\mathcal{C}$-injective if and only if both $X$ and $Y$ are.
\end{lemma}

It is known that a right module $P$ over a ring $R$ is finitely generated projective if and only if the canonical map $M \otimes_R \Hom_R(P, R) \to \Hom_R(P, M)$ is bijective for every right $R$-module $M$. An analogous result is formulated and proved for an algebra $A$ in $\mathcal{C}$ as follows: For two objects $M, P \in \mathcal{C}_A$, we define
\begin{equation}
  \label{eq:C-proj-isomorphism}
  \Theta_{M, P} : M \otimes_A \iHom_A(P, A) \to \iHom_A(P, M)
\end{equation}
to be the morphism in $\mathcal{C}$ induced by
\begin{equation*}
  M \otimes \iHom_A(P, A)
  \xrightarrow{\ \mathfrak{a}_{M,P,A} \ }
  \iHom_A(P, M \otimes A)
  \xrightarrow{\ \iHom_A(P, \act_M) \ } \iHom_A(P, M).
\end{equation*}

\begin{lemma}
  For $P \in \mathcal{C}_A$, the following are equivalent.
  \begin{enumerate}
  \item The morphism $\Theta_{M,P}$ is invertible for all $M \in \mathcal{C}_A$.
  \item $P$ is $\mathcal{C}$-projective.
  \end{enumerate}
\end{lemma}
\begin{proof}
  The morphism $\Theta_{M,P}$ is natural in $M \in \mathcal{C}_A$.
  Thus, if (1) holds, then the functor $\iHom_A(P, -)$ is right exact. Hence (1) implies (2). The converse follows from the Eilenberg-Watts equivalence.
\end{proof}

It is well-known that a module over a ring is projective if and only if it is a direct summand of a free module.
Here we give the following characterization of $\mathcal{C}$-projectivity by free modules and its dual statement:

\begin{lemma}
  \label{lem:C-proj-and-free}  
  Let $A$ be an algebra in $\mathcal{C}$. For $M \in \mathcal{C}_A$, we have:
  \begin{enumerate}
  \item [(a)] $M$ is $\mathcal{C}$-projective if and only if there are non-zero objects $X$ and $Y$ of $\mathcal{C}$ such that the right $A$-module $X \otimes M$ is a direct summand of $Y \otimes A$.
  \item [(b)] $M$ is $\mathcal{C}$-injective if and only if there are non-zero objects $X$ and $Y$ of $\mathcal{C}$ such that the right $A$-module $X \otimes M$ is a direct summand of $Y \otimes A^*$.
  \end{enumerate}
\end{lemma}
\begin{proof}
  (a) Suppose that $M \in \mathcal{C}_A$ is $\mathcal{C}$-projective. We take a non-zero projective object $X \in \mathcal{C}$ and set $Y = X \otimes M$ (regarded as an object of $\mathcal{C}$). We consider the epimorphism $q := \id_X \otimes \rho_M : Y \otimes A \to X \otimes M$ in $\mathcal{C}_A$.
  Since $M$ is $\mathcal{C}$-projective, the target of $q$ is projective in $\mathcal{C}_A$, and thus $q$ splits. Therefore $X \otimes M$ is a direct summand of $Y \otimes A$.
  The converse is easily verified by Lemmas~\ref{lem:C-proj} and \ref{lem:C-proj-direct-summands} and the fact that $A \in \mathcal{C}_A$ is $\mathcal{C}$-projective.

  (b) Suppose that $M$ is $\mathcal{C}$-injective.
  We take a non-zero injective object $X \in \mathcal{C}$ and set $Y = X \otimes M$.
  Then there is a monomorphism $\id_X \otimes \rho^{\natural}_M : X \otimes M \to Y \otimes A^*$ in $\mathcal{C}_A$, where $\rho^{\natural}_M : M \to M \otimes A^*$ is the morphism induced by the right action of $A$ on $M$ (see Example~\ref{ex:internal-Hom-C-A}).
  Thus $X \otimes M$ becomes a direct summand of $Y \otimes A^*$.
  The `only if' part has been verified. The converse follows from Lemmas~\ref{lem:C-inj} and \ref{lem:C-inj-direct-summands} and the fact that $A^* \in \mathcal{C}_A$ is $\mathcal{C}$-projective.
\end{proof}

We do not know whether we can take $X$ to be $\unitobj$ in (a) and (b) in Lemma~\ref{lem:C-proj-and-free}, that is, whether every $\mathcal{C}$-projective (respectively, $\mathcal{C}$-injective) object of $\mathcal{C}_A$ is a direct summand of $Y \otimes A$ (respectively, $Y \otimes A^*$) for some $Y \in \mathcal{C}$.
However, we can prove a weaker statement:

\begin{lemma}
  \label{lem:C-proj-and-free-2}
  Let $A$ be an algebra in $\mathcal{C}$.
  \begin{enumerate}
  \item [(a)] Every $\mathcal{C}$-projective object of $\mathcal{C}_A$ is a subobject of $X \otimes A$ for some $X \in \mathcal{C}$.
  \item [(b)] Every $\mathcal{C}$-injective object of $\mathcal{C}_A$ is a quotient of $X \otimes A^*$ for some $X \in \mathcal{C}$.
  \end{enumerate}
\end{lemma}
\begin{proof}
  We prove (a).
  Let $P$ be a $\mathcal{C}$-projective object of $\mathcal{C}_A$.
  By Lemma~\ref{lem:C-inj-direct-summands}, there are non-zero objects $V, W \in \mathcal{C}$ and a split monomorphism $i: V \otimes P \to W \otimes A$ in $\mathcal{C}_A$.
  Letting $X = {}^*V \otimes W$, we have a monomorphism
  \begin{equation*}
    P \xrightarrow{\quad \coev \otimes \id \quad}
    {}^*V \otimes V \otimes P
    \xrightarrow{\quad \id \otimes i \quad} {}^*V \otimes W \otimes A = X \otimes A
  \end{equation*}
  in $\mathcal{C}_A$. The proof of (a) is done. Part (b) is proved in a similar way.
\end{proof}

In view of the definition of a quasi-Frobenius algebra in $\mathcal{C}$, it would be important to know when a $\mathcal{C}$-projective object is $\mathcal{C}$-injective.
Here we provide the following criteria:

\begin{lemma}
  \label{lem:C-proj-and-free-3}
  Let $A$ be an algebra in $\mathcal{C}$.
  \begin{enumerate}
  \item [(a)] Let $E$ be a $\mathcal{C}$-injective object of $\mathcal{C}_A$.
    Then $E$ is $\mathcal{C}$-projective if and only if there is a monomorphism $E \to X \otimes A$ in $\mathcal{C}_A$ for some $X \in \mathcal{C}$.
  \item [(b)] Let $P$ be a $\mathcal{C}$-projective object of $\mathcal{C}_A$.
    Then $P$ is $\mathcal{C}$-injective if and only if there is an epimorphism $X \otimes A^* \to P$ in $\mathcal{C}_A$ for some $X \in \mathcal{C}$.
  \end{enumerate}
\end{lemma}
\begin{proof}
  We prove (a).
  Let $E$ be a $\mathcal{C}$-injective object of $\mathcal{C}_A$.
  The `only if' part of (a) follows from Lemma~\ref{lem:C-proj-and-free-2} (without the assumption that $E$ is $\mathcal{C}$-injective).
  To prove the `if' part, we assume that there are an object $X \in \mathcal{C}$ and a monomorphism $i: E \to X \otimes A$ in $\mathcal{C}_A$.
  We take a non-zero injective object $J \in \mathcal{C}$.
  Since $J \otimes E$ is injective in $\mathcal{C}_A$ by Lemma~\ref{lem:C-inj}, the monomorphism $\id_J \otimes i : J \otimes E \to J \otimes X \otimes A$ splits in $\mathcal{C}_A$.
  Hence, by Lemmas~\ref{lem:C-proj} and \ref{lem:C-proj-direct-summands}, $E$ is $\mathcal{C}$-projective.
  The proof of (a) is done. Part (b) is proved in a similar way.
\end{proof}

\subsection{Relative Serre functor of $\mathcal{C}_A$}
\label{subsec:rel-Serre-A-mod}

Let $A$ be an algebra in $\mathcal{C}$.
In Lemma \ref{lem:rel-Serre-mod-A}, we have given a formula of a relative Serre functor of $\mathcal{C}_A$.
In this subsection, we discuss a relative Serre functor of $\mathcal{C}_A$ and related morphisms in detail.

We define the functor $\Ser_A : \mathcal{C}_A \to \mathcal{C}_A$ by \eqref{eq:rel-Serre-mod-A}.
For $X \in \mathcal{C}$ and $M \in \mathcal{C}_A$, there is a natural isomorphism
\begin{equation*}
  \mathfrak{s}^{(A)}_{X,M} := (\mathfrak{b}_{M,A,X}^{-1})^*:
  X^{**} \otimes \Ser_A(M) \to \Ser_A(X \otimes M).
\end{equation*}
For $P \in \relprj{\mathcal{C}}{\mathcal{C}_A}$, we define
\begin{equation*}
  \itrace^{(A)}_P := \overline{\eval}_{\iHom_A(P, A)} \circ \Theta_{\Ser_A(P), P}^{-1} : \iHom_A(P, \Ser_A(P)) \to \unitobj,
\end{equation*}
where $\overline{\eval}_L : L^* \otimes_A L \to \unitobj$ for $L \in \mathcal{C}_A$ is the morphism in $\mathcal{C}$ induced by the evaluation morphism $\eval_L : L^* \otimes L \to \unitobj$ and $\Theta_{M,P}$ for $M, P \in \mathcal{C}_A$ is the natural transformation \eqref{eq:C-proj-isomorphism}. We also define
\begin{equation*}
  \phi^{(A)}_{M,P} : \iHom_A(M, \Ser_A(P)) \to \iHom_A(P, M)^*
  \quad (P \in \relprj{\mathcal{C}}{\mathcal{C}_A}, M \in \mathcal{M})
\end{equation*}
to be the morphism in $\mathcal{C}$ induced by the pairing
\begin{equation*}
  \itrace^{(A)}_P \circ \icomp_{P,M,\Ser_A(P)}: \iHom_A(M, \Ser_A(P)) \otimes \iHom_A(P, M) \to \unitobj.
\end{equation*}

\begin{theorem}
  \label{thm:rel-Serre-A-mod-structured}
  The triple $(\Ser_A, \mathfrak{s}^{(A)}, \phi^{(A)})$ constructed in the above is a structured relative Serre functor of $\mathcal{C}_A$ in the sense of Definition \ref{def:rel-Serre-structured}.
  The trace associated to this triple is $\itrace^{(A)}$.
\end{theorem}
\begin{proof}
  We fix an object $P \in \relprj{\mathcal{C}}{\mathcal{C}_A}$.
  The functor $\T_P^{\radj}$ has a right adjoint
  \begin{equation*}
    \T_P^{\rradj} = (-) \otimes \iHom_A(P, A)^* = (-) \otimes \Ser_A(P).
  \end{equation*}
  Indeed, in a similar way as the proof of Lemma~\ref{lem:rel-Serre-mod-A}, we have
  \begin{align*}
    \mathfrak{H} := \Hom_{\mathcal{C}_A}(M, X \otimes \iHom_A(P, A)^*)
    & = \Hom_{\mathcal{C}_A}(M, \iHom_{\mathcal{C}}(\iHom_A(P, A), X)) \\
    & \cong \Hom_{\mathcal{C}}(M \otimes_A \iHom_A(P, A), X) \\
    & \cong \Hom_{\mathcal{C}}(\iHom_A(P, M), X)
  \end{align*}
  for $M \in \mathcal{C}_A$ and $X \in \mathcal{C}$.
  We note that an element $f \in \mathfrak{H}$ is mapped to
  \begin{equation*}
    (\id_X \otimes \overline{\eval}_{\iHom_A(P,A)}) \circ (f \otimes_A \id) \circ \Theta_{M,P}^{-1} : \iHom_A(P, M) \to X
  \end{equation*}
  by the above natural isomorphisms. This implies that the morphism $\itrace^{(A)}_P$ is equal to the $\unitobj$-component of the counit of the adjunction $\T_P^{\radj} \dashv \T_P^{\rradj}$.

  Now we equip the restriction of $\Ser_A$ to $\relprj{\mathcal{C}}{\mathcal{C}_A}$ with a twisted left $\mathcal{C}$-module structure in the way as we did in Subsection \ref{subsec:dual-internal-hom}, and denote by
  \begin{equation*}
    \mathfrak{s}_{X,P} : X^{**} \otimes \Ser_A(P) \to \Ser_A(X \otimes P)
    \quad (X \in \mathcal{C}, P \in \relprj{\mathcal{C}}{\mathcal{C}_A})
  \end{equation*}
  the structure morphism.
  Since the left $\mathcal{C}$-module structure of $\T_P^{\rradj}$ is the identity (by our assumption that every monoidal category is strict), $\mathfrak{s}_{X,P}$ is equal to
  \begin{equation*}
    X^{**} \otimes \Ser_A(P) = (\T_X^{\rradj} \circ \T_P^{\rradj})(\unitobj)
    \xrightarrow{\quad \eqref{eq:FG-r-adj} \quad}
    (\T_X \circ \T_P)^{\rradj}(\unitobj) = \Ser_A(X \catactl P).
  \end{equation*}
  Hence, by the construction of $\mathfrak{b}$, the equation $\mathfrak{s}_{X,P} = \mathfrak{s}^{(A)}_{X,P}$ holds for all objects $X \in \mathcal{C}$ and $P \in \relprj{\mathcal{C}}{\mathcal{C}_A}$.
  The proof is completed by Lemmas \ref{lem:rel-Serre-phi-M-P} and \ref{lem:rel-Serre-trace-1}.
\end{proof}

The following description of $\Ser_A$ is convenient:

\begin{corollary}
  \label{cor:rel-Serre-A-mod-tensor}
  There are isomorphisms
  \begin{equation*}
    \Ser_A(M) \cong (M \otimes_A {}^* \! A)^{**}
    \cong M^{**} \otimes_{A^{**}} A^*
    \quad (M \in \mathcal{C}_A)
  \end{equation*}
  of twisted left $\mathcal{C}$-module functors.
\end{corollary}
\begin{proof}
  We define the functor $F: \mathcal{C}_A \to \mathcal{C}_{\,{}^{**}\!A}$ by $F(M) = {}^{**}(\Ser_A(M))$ for $M \in \mathcal{C}_A$. Then we have $F(A) \cong {}^{**}(A^*) \cong {}^{*} \! A$ as $A$-${}^{**}\!A$-bimodules in $\mathcal{C}$.
  The Eilenberg-Watts equivalence shows that $F$ is isomorphic to $(-) \otimes_A {}^{*}\!A$ as left $\mathcal{C}$-module functors.
  Hence we obtain the first isomorphism.
  The second isomorphism is obtained by considering the functor $M \mapsto \Ser_A({}^{**}\!M)$ instead of $F$ ({\it cf}. the proof Lemma \ref{lem:rel-Serre-mod-A-ra}).
\end{proof}

\subsection{Frobenius algebras}

Let $A$ be an algebra in $\mathcal{C}$ with multiplication $m_A$ and unit $u_A$.
A {\em Frobenius form} on $A$ is a morphism $\lambda : A \to \unitobj$ in $\mathcal{C}$ such that
\begin{equation}
  \label{eq:Fb-alg-phi}
  (\lambda m_A \otimes \id_{A^*}) \circ (\id_A \otimes \coev_A) : A \to A^*
\end{equation}
is an isomorphism in $\mathcal{C}$. A {\em Frobenius algebra} in $\mathcal{C}$ is an algebra in $\mathcal{C}$ equipped with a Frobenius form.
In this subsection, we give a formula of a relative Serre functor for a Frobenius algebra and note its consequences.

Let $A$ be a Frobenius algebra with Frobenius form $\lambda_A$.
We define the {\em Nakayama isomorphism} $\nu_A$ of $A$ with respect to $\lambda_A$ by
\begin{equation}
  \label{eq:Nakayama-iso-def}
  \nu_A := \phi^{-1} \circ \phi^{*} : A^{**} \to A,
\end{equation}
where $\phi : A \to A^*$ is the isomorphism \eqref{eq:Fb-alg-phi} with $\lambda = \lambda_A$.
This term is justified as follows:
For $V, W \in \mathcal{C}$, there is a linear isomorphism
\begin{equation}
  \label{eq:Nakayama-iso-D}
  \begin{gathered}
    D_{V,W} : \Hom_{\mathcal{C}}(V \otimes W, \unitobj) \to \Hom_{\mathcal{C}}(W^{**} \otimes V, \unitobj), \\
    b \mapsto \eval_{W^*} \circ (\id_{W^{**}} \otimes b \otimes \id_{V^*}) \circ (\id_{W^{**}} \otimes \id_V \otimes \coev_W).
  \end{gathered}
\end{equation}
By the definition of $\phi$, we have
\begin{equation}
  \label{eq:Nakayama-iso-def-2}
  \begin{gathered}
    D_{A,A}(\lambda_A \circ m_A)
    = \eval_{A^*} \circ (\id_{A^{**}} \otimes \phi)
    = \eval_{A} \circ (\phi^* \otimes \id_{A}) \\
    = \eval_{A} \circ ((\phi \circ \nu_A) \otimes \id_{A})
    = \lambda_A \circ m_A \circ (\nu_A \otimes \id_{A}).
  \end{gathered}
\end{equation}
This means that, if $\mathcal{C}$ has a pivotal structure $\mathfrak{p}_X : X \to X^{**}$ ($X \in \mathcal{C}$), then $\nu_A \circ \mathfrak{p}_A$ is equal to the Nakayama automorphism of $A$ introduced in \cite{MR2500035}.

By the same way as \cite{MR2500035}, one can show that $\nu_A$ is an isomorphism of algebras in $\mathcal{C}$ even in the absence of a pivotal structure.
Given a morphism $f: R \to S$ of algebras in $\mathcal{C}$, we denote by $(-)_{(f)} : {}_S\mathcal{C} \to {}_R\mathcal{C}$ and ${}_{(f)}(-) : \mathcal{C}_S \to \mathcal{C}_R$ the functors induced by $f$. The relative Serre functor of the category of modules over a Frobenius algebra in $\mathcal{C}$ is given as follows:

\begin{lemma}
  \label{lem:Fb-alg-rel-Serre}
  Let $A$ be a Frobenius algebra in $\mathcal{C}$ with Nakayama isomorphism $\nu$, and let $\Ser$ be a relative Serre functor of $\mathcal{C}_A$. Then there is an isomorphism
  \begin{equation*}
    \Ser(M) \cong (M^{**})_{(\nu^{-1})} \quad (M \in \mathcal{C}_A)
  \end{equation*}
  of twisted left $\mathcal{C}$-module functors. In particular, $\Ser$ is an autoequivalence.
\end{lemma}
\begin{proof}
  Let $m : A \otimes A \to A$ be the multiplication of $A$, and let $\lambda : A \to \unitobj$ be the Frobenius form of $A$. We define $\phi : A \to A^*$ by \eqref{eq:Fb-alg-phi}.
  It is easy to see that $\phi$ is a morphism of right $A$-modules in $\mathcal{C}$. We also have
  \begin{align*}
    \eval_{A^*} \circ (m^{**} \otimes \phi)
    & = \lambda \circ m \circ (\nu \otimes \id) \circ (m^{**} \otimes \id_A) \\
    & = \lambda \circ m \circ (m \otimes \id) \circ (\nu \otimes \nu \otimes \id_A) \\
    & = \lambda \circ m \circ (\id \otimes m) \circ (\nu \otimes \nu \otimes \id_A) \\
    & = \eval_{A^*} \circ (\id_{A^{**}} \otimes (\phi \circ m \circ (\nu \otimes \id_A))),
  \end{align*}
  where~\eqref{eq:Nakayama-iso-def-2} is used at the first and the last equality.
  This means that $\phi$ is an isomorphism ${}_{(\nu)}A \to A^*$ of left $A^{**}$-modules in $\mathcal{C}$.

  By the above discussion, we find that $\phi^* : (A^{**})_{(\nu^{-1})} \to A^*$ is an isomorphism of $A^{**}$-$A$-bimodules in $\mathcal{C}$. By Corollary \ref{cor:rel-Serre-A-mod-tensor}, we have isomorphisms
  \begin{equation*}
    \Ser(M) \cong M^{**} \otimes_{A^{**}} A^*
    \cong M^{**} \otimes_{A^{**}}(A^{**})_{(\nu^{-1})}
    \cong (M^{**})_{(\nu^{-1})}
    \quad (M \in \mathcal{C}_A)
  \end{equation*}
  of twisted left $\mathcal{C}$-module functors. The proof is done.
\end{proof}

The relative Serre functor is useful for constructing Frobenius algebras:

\begin{lemma}
  \label{lem:Fb-alg-construction}
  Let $\mathcal{M}$ be a finite left $\mathcal{C}$-module category with structured relative Serre functor $(\Ser, \mathfrak{s}, \phi)$, and let $P$ be a $\mathcal{C}$-projective object of $\mathcal{M}$.
  If $q: P \to \Ser(P)$ is an isomorphism in $\mathcal{M}$, then the morphism
  \begin{equation*}
    \lambda := \itrace_{P} \circ \iHom(P, q) : \iEnd(P) \to \unitobj
  \end{equation*}
  is a Frobenius form of $\iEnd(P)$, where $\itrace$ is the trace associated to $(\Ser, \mathfrak{s}, \phi)$.
\end{lemma}
\begin{proof}
  We write $A = \iEnd(P)$ and $m = \icomp_{P,P,P}$. By \eqref{eq:itrace-1}, we have
  \begin{equation*}
    (\lambda m \otimes \id_{A^*}) \circ (\id_A \otimes \coev_A)
    = \phi_{P,P} \circ \iHom(P, q).
  \end{equation*}
  The right-hand side is an isomorphism in $\mathcal{C}$. The proof is done.
\end{proof}

By the construction of a Frobenius algebra given in Lemma~\ref{lem:Fb-alg-construction}, we obtain the following characterization of a finite left $\mathcal{C}$-module category that is equivalent to the category of modules over a Frobenius algebra.

\begin{lemma}
  \label{lem:Fb-alg-modules}
  For a finite left $\mathcal{C}$-module category $\mathcal{M}$, the following are equivalent:
  \begin{enumerate}
  \item The relative Serre functor $\Ser$ of $\mathcal{M}$ is an equivalence.
  \item There exists a $\mathcal{C}$-progenerator $G \in \mathcal{M}$ such that $\Ser(G) \cong G$.
  \item There exists a Frobenius algebra $A$ in $\mathcal{C}$ such that $\mathcal{M}$ is equivalent to $\mathcal{C}_A$ as a left $\mathcal{C}$-module category.
  \end{enumerate}
\end{lemma}
\begin{proof}
  Suppose that (1) holds.
  Let $\{ P_1, \cdots, P_n \}$ be the complete representatives of the isomorphism classes of indecomposable projective objects of $\mathcal{M}$. Since $\Ser$ is assumed to be an autoequivalence of $\mathcal{M}$, it permutes the set $\{ P_1, \cdots, P_n \}$ up to isomorphism. Thus $G = P_1 \oplus \dotsb \oplus P_n$ is a $\mathcal{C}$-generator such that $\Ser(G) \cong G$. We have proved that (1) implies (2).

  If there is an object $G \in \mathcal{M}$ as in (2), then $\mathcal{M} \approx \mathcal{C}_A$ as left $\mathcal{C}$-module categories, where $A = \iEnd(G)$. Since $A$ is Frobenius by Lemma \ref{lem:Fb-alg-construction}, (3) holds. Finally, if (3) holds, then $\Ser$ is an equivalence by Lemma~\ref{lem:Fb-alg-rel-Serre}. The proof is done.
\end{proof}

\subsection{Quasi-Frobenius algebras}

We now give equivalent conditions for an algebra in $\mathcal{C}$ to be quasi-Frobenius as follows:

\begin{theorem}
  \label{thm:QF-algebras}
  Let $A$ be an algebra in $\mathcal{C}$, and let $\Ser$ be the relative Serre functor of $\mathcal{C}_A$. Then the following are equivalent:
  \begin{enumerate}
  \item The right $A$-module $A \in \mathcal{C}_A$ is $\mathcal{C}$-injective,
    that is, $A$ is quasi-Frobenius.
  \item The right $A$-module $A^* \in \mathcal{C}_A$ is $\mathcal{C}$-projective.
  \item All $\mathcal{C}$-projective objects of $\mathcal{C}_A$ are $\mathcal{C}$-injective.
  \item All $\mathcal{C}$-injective objects of $\mathcal{C}_A$ are $\mathcal{C}$-projective.
  \item The functor $\Ser$ is an equivalence.
  \item The functor $\Ser$ is exact.
  \item A right adjoint of the functor $\Ser$ is exact.
  \item The Nakayama functor of $\mathcal{C}_A$ is an equivalence.
  \item The algebra $A$ is Morita equivalent to a Frobenius algebra in $\mathcal{C}$.
  \end{enumerate}
\end{theorem}

We have mainly considered finite left module categories.
One can define the internal Hom functor, $\mathcal{C}$-projective objects, $\mathcal{C}$-injective objects and the relative Serre functor of a finite right $\mathcal{C}$-module category and formulate the left module versions of the conditions (1)--(8).
Since the condition (9) is left-right symmetric, the left module versions of (1)--(8) are also equivalent to that $A$ is quasi-Frobenius.

Although we will not state them one by one, characterizations of $\mathcal{C}$-projective objects and $\mathcal{C}$-injective objects given in Subsection~\ref{subsec:C-proj-C-inj} yield more conditions for an algebra in $\mathcal{C}$ to be quasi-Frobenius.
We also note that some conditions equivalent to (8) are found in Subsection \ref{subsec:nakayama} and references therein.

Theorem \ref{thm:QF-algebras} is well known for $\mathcal{C} = \Vect$.
Our argument gives no new proof for the known characterization of quasi-Frobenius algebras in $\Vect$. The proof presented below uses a property of the Nakayama functor, which essentially depends on known results on (quasi-)Frobenius algebras in $\Vect$.

To prove Theorem~\ref{thm:QF-algebras}, in addition to lemmas on $\mathcal{C}$-projective and $\mathcal{C}$-injective objects given in Subsection~\ref{subsec:C-proj-C-inj}, we use the following technical lemma:

\begin{lemma}
  \label{lem:proj-C-inj-is-inj}
  Let $M$ be an object of a finite left $\mathcal{C}$-module category $\mathcal{M}$. Then the following hold:
  \begin{enumerate}
  \item[(a)] If $M$ is projective and $\mathcal{C}$-injective, then $M$ is injective.
  \item[(b)] If $M$ is injective and $\mathcal{C}$-projective, then $M$ is projective.
  \end{enumerate}
\end{lemma}
\begin{proof}
  We only prove (a), since (b) is proved in a similar way.
  We suppose that $M$ is projective and $\mathcal{C}$-injective.
  Let $E$ be a non-zero injective object of $\mathcal{C}$.
  Since $E^* \otimes E \in \mathcal{C}$ is injective, the object $J := (E^* \otimes E) \catactl M \in \mathcal{M}$ is injective.
  The evaluation morphism induces an epimorphism $J \to M$ in $\mathcal{M}$, which splits since $M$ is projective. Thus $M$ is injective as a direct summand of $J$.
\end{proof}

\begin{proof}[Proof of Theorem~\ref{thm:QF-algebras}]
  We define endofunctors $\Ser_A$ and $\overline{\Ser}_A$ on $\mathcal{C}_A$ by
  \begin{equation*}
    \Ser_A(M) = \iHom_A(M,A)^*
    \quad \text{and} \quad
    \overline{\Ser}_A(M) = {}^{**}\iHom_A(A^*, M)
    \quad (M \in \mathcal{C}_A),
  \end{equation*}
  respectively. By Lemmas~\ref{lem:rel-Serre-mod-A} and \ref{lem:rel-Serre-mod-A-ra}, the former is a relative Serre functor of $\mathcal{C}_A$ and the latter one is right adjoint to $\Ser_A$.
  By Theorem \ref{thm:rel-Serre-summary} (\ref{item:rel-Serre-main-2}) and the invertibility of $\Nak_{\mathcal{C}}(\unitobj)$, conditions (5) and (8) are equivalent.
  Applying Lemma \ref{lem:Fb-alg-modules} to $\mathcal{M} = \mathcal{C}_A$, we see that (5) and (9) are equivalent.
  It is obvious that (5) implies (6) and (7).
  The remaining part of this theorem is proved via the following route:
  \begin{equation*}
    (6) \Rightarrow (1) \Rightarrow (3) \Rightarrow (8), \quad
    (7) \Rightarrow (2) \Rightarrow (4) \Rightarrow (8).
  \end{equation*}

  (6) $\Rightarrow$ (1). This follows from the defining formula of $\Ser_A$.

  \smallskip
  (1) $\Rightarrow$ (3).
  Suppose that (1) holds. Let $M$ be a $\mathcal{C}$-projective object of $\mathcal{C}_A$.
  Then, by Lemma~\ref{lem:C-proj}, there are non-zero objects $X$ and $Y$ of $\mathcal{C}$ such that $X \otimes M$ is a direct summand of $Y \otimes A$.
  Since $A$ is $\mathcal{C}$-injective by the assumption, $X \otimes M$ is also $\mathcal{C}$-injective by Lemma \ref{lem:C-inj-direct-summands}. Thus, by Lemma \ref{lem:C-inj}, $M$ is $\mathcal{C}$-injective.

  \smallskip
  (3) $\Rightarrow$ (8). Suppose that (3) holds.
  Let $P$ be a projective object of $\mathcal{C}_A$. Then $P$ is also $\mathcal{C}$-injective by Lemma~\ref{lem:relative-proj-1} and the assumption.
  Thus, by Lemma~\ref{lem:proj-C-inj-is-inj}, $P$ is injective. In summary, we have proved that all projective objects of $\mathcal{C}_A$ are injective. By a property of the Nakayama functor recalled in Subsection \ref{subsec:nakayama}, we conclude that $\Nak_{\mathcal{C}_A}$ is an equivalence.

  \smallskip
  $(7)$ $\Rightarrow$ (2). This follows from the defining formula of $\overline{\Ser}_A$.

  \smallskip
  It remains to show (2) $\Rightarrow$ (4) and (4) $\Rightarrow$ $(8)$.
  These are proved in a similar way as (1) $\Rightarrow$ (3) and (3) $\Rightarrow$ (8), respectively, with the help of lemmas on $\mathcal{C}$-projective and $\mathcal{C}$-injective objects exhibited in Subsection~\ref{subsec:C-proj-C-inj}. The proof is done.
\end{proof}

We give some corollaries of this theorem. As one might expect,

\begin{corollary}
  All Frobenius algebras in $\mathcal{C}$ are quasi-Frobenius.
\end{corollary}
\begin{proof}
  This is a trivial consequence of Theorem \ref{thm:QF-algebras}.
\end{proof}

Let $F: \mathcal{C} \to \mathcal{D}$ be a tensor functor \cite[Definition 4.2.5]{MR3242743} from $\mathcal{C}$ to another finite tensor category $\mathcal{D}$, and let $A$ be a quasi-Frobenius algebra in $\mathcal{C}$.
Then the object $B := F(A)$ has a natural structure of an algebra in $\mathcal{D}$.
By Lemma~\ref{lem:C-proj-and-free} and Theorem \ref{thm:QF-algebras}, there are non-zero objects $X$ and $Y$ such that $X \otimes A$ is a direct summand of $Y \otimes A^*$ in $\mathcal{C}_A$. By applying the functor $F$, we see that $F(X) \otimes B$ is a direct summand of $F(Y) \otimes B^*$ in $\mathcal{D}_B$. Thus $B$ is quasi-Frobenius. We have proved:

\begin{corollary}
  Quasi-Frobenius algebras are preserved by tensor functors between finite tensor categories.
\end{corollary}

It would be convenient if we rewrite our results on quasi-Frobenius algebras as a theorem on finite left $\mathcal{C}$-module categories. The following corollary, which upgrades Lemma~\ref{lem:Fb-alg-modules}, is obtained just by packaging our results:

\begin{corollary}
  \label{cor:modules-over-QF-alg}
  For a finite left $\mathcal{C}$-module category $\mathcal{M}$, the following are equivalent:
  \begin{enumerate}
  \item All $\mathcal{C}$-projective objects of $\mathcal{M}$ are $\mathcal{C}$-injective.
  \item All $\mathcal{C}$-injective objects of $\mathcal{M}$ are $\mathcal{C}$-projective.
  \item The relative Serre functor of $\mathcal{M}$ is an equivalence.
  \item The Nakayama functor of $\mathcal{M}$ is an equivalence.
  \item There exists a quasi-Frobenius algebra $A$ in $\mathcal{C}$ such that $\mathcal{M} \approx \mathcal{C}_A$.
  \item There exists a Frobenius algebra $A$ in $\mathcal{C}$ such that $\mathcal{M} \approx \mathcal{C}_A$.
  \end{enumerate}
\end{corollary}

Lemma~\ref{lem:C-proj} implies that an exact left $\mathcal{C}$-module category is nothing but a finite left $\mathcal{C}$-module category where every object is $\mathcal{C}$-projective. In view of the importance of exact module categories, the following corollary could be noteworthy:

\begin{corollary}
  For an exact left $\mathcal{C}$-module category $\mathcal{M}$, there exists a Frobenius algebra $A$ in $\mathcal{C}$ such that $\mathcal{M} \approx \mathcal{C}_A$.
\end{corollary}

\subsection{Hopf algebras}

It is well-known that a finite-dimensional Hopf algebra is a Frobenius algebra.
We suppose that $\mathcal{C}$ has a braiding so that the notion of a Hopf algebra in $\mathcal{C}$ makes sense. As we will see in Example~\ref{ex:non-Frobenius-Hopf}, a Hopf algebra in $\mathcal{C}$ is not necessarily a Frobenius algebra in $\mathcal{C}$. Nevertheless, the following theorem holds:

\begin{theorem}
  Every Hopf algebra in $\mathcal{C}$ is quasi-Frobenius.
\end{theorem}
\begin{proof}
  The integral theory for Hopf algebras \cite{MR1685417,MR1759389} is essential for proving this theorem.
  Let $H$ be a Hopf algebra in $\mathcal{C}$, and let $\lambda: H \to I$ be the universal right integral on $H$, where $I$ is the object of integrals of $H$ in the sense of \cite{MR1759389}.
  It is known that $\lambda$ induces an isomorphism $H \cong I \otimes H^*$ in $\mathcal{C}_H$ (see \cite[Subsection 4.2]{MR1759389}). Hence the right $H$-module $H$ is $\mathcal{C}$-injective. Thus $H$ is quasi-Frobenius.
\end{proof}

\begin{example}
  \label{ex:non-Frobenius-Hopf}
  Let $V$ be a finite-dimensional vector space of dimension $m$.
  The exterior algebra $H := \bigoplus_{j = 0}^m \bigwedge^j V$ has a natural structure of a Hopf algebra in the category $\mathbf{SVec}$ of finite-dimensional supervector spaces with supersymmetry.
  The universal right integral on $H$ is the projection $\pi : H \to I$, where $I := \bigwedge^m V$.

  The algebra $H$ is a Frobenius algebra in $\mathbf{SVec}$ if and only if $m$ is even. To see this, we note that $I$ is isomorphic to the unit object of $\mathbf{SVec}$ if and only if $m$ is even. Thus, when $m$ is even, $\pi$ can be viewed as a Frobenius form on $H$. On the other hand, when $m$ is odd, every morphism $H \to \unitobj$ in $\mathbf{SVec}$ vanishes on the ideal $I$ of $H$. This implies that $H$ has no Frobenius form.
\end{example}

\subsection{A conjecture on simple algebras}

We say that an algebra $A$ in $\mathcal{C}$ is {\em simple} if it is a simple object in ${}_A\mathcal{C}_A$.
It is well-known that a simple algebra in $\Vect$ is isomorphic to the matrix algebra over a finite-dimensional division algebra over $\bfk$ and, in particular, is a Frobenius algebra. In general, a simple algebra in $\mathcal{C}$ is not Frobenius as the following example shows:

\begin{example}
  \label{ex:non-Frobenius-exact}
  Suppose that the base field $\bfk$ is algebraically closed. Let $\mathcal{D}$ be a finite tensor category, and let $\mathcal{C} = \mathcal{Z}(\mathcal{D})$ be the Drinfeld center of $\mathcal{D}$ \cite[Section 7.13]{MR3242743}. The forgetful functor $U: \mathcal{C} \to \mathcal{D}$ has a right adjoint,
  the object $A := U^{\radj}(\unitobj)$ is a commutative algebra in $\mathcal{C}$ in a natural way,
  and there is an equivalence $K: \mathcal{D} \to \mathcal{C}_A$ induced by $U^{\radj}$.
  Since $A = K(\unitobj)$, the right $A$-module $A$ is simple.
  Hence, in particular, $A$ is a simple algebra in $\mathcal{C}$.
  According to \cite[Section 6]{MR3632104}, the algebra $A$ admits a Frobenius form if and only if $\mathcal{D}$ is unimodular in the sense of \cite[Section 6.5]{MR3242743}. Thus, if $\mathcal{D}$ is not unimodular, $A$ is a simple commutative algebra in $\mathcal{C}$ which is not Frobenius.
\end{example}

We go back to the general setting and raise the following:

\begin{conjecture}
  \label{conj:simple->QF}
  A simple algebra in $\mathcal{C}$ is quasi-Frobenius.
\end{conjecture}

An algebra $A$ in $\mathcal{C}$ is said to be {\em exact} if $\mathcal{C}_A$ is an exact left $\mathcal{C}$-module category.
From now on, we again assume that $\bfk$ is algebraically closed.
Etingof and Ostrik conjectured that every simple algebra in $\mathcal{C}$ is exact \cite[Conjecture B.6]{MR4237968}.
Conjecture \ref{conj:simple->QF} implies their conjecture.
Indeed, Etingof and Ostrik have also proved that a simple algebra $A$ in $\mathcal{C}$ is exact provided that
\begin{equation}
  \label{eq:assumption-in-EO}
  \text{there are an object $X \in \mathcal{C}$ and a monomorphism $A^* \to X \otimes A$ in $\mathcal{C}_A$}
\end{equation}
\cite[Theorem B.1]{MR4237968}\footnote{The proof of \cite[Theorem B.1]{MR4237968} does not work without the assumption that the base field is algebraically closed since it relies on the theory of the Frobenius-Perron dimension.}.
By Lemma~\ref{lem:C-proj-and-free-3}, the condition \eqref{eq:assumption-in-EO} is equivalent to that $A$ is quasi-Frobenius.

Now we assume that $A$ is simple.
Then, as noted after \cite[Remark B.5]{MR4237968}, the condition \eqref{eq:assumption-in-EO} is equivalent to $A^* \otimes_A {}^* \! A \ne 0$.
The object of the left-hand side is written in terms of the relative Serre functor.
Indeed, let $\Ser_A$ be the relative Serre functor of $\mathcal{C}_A$ given by Lemma~\ref{lem:rel-Serre-mod-A-ra}. Then we have
\begin{equation*}
  {}^{**}(\Ser_A^2(A)) = {}^{*}\iHom_A(\iHom_A(A, A)^*, A)
  \cong {}^*\iHom_A(A^*, A)
  \cong A^* \otimes_A {}^*\!A
\end{equation*}
by Lemma~\ref{lem:iHom-dual}. Thus, in view of the Eilenberg-Watts equivalence, we can say that \eqref{eq:assumption-in-EO} is equivalent to the condition $\Ser_A^2 \ne 0$.

\section{Symmetric Frobenius algebras}
\label{sec:symm-frobenius}

Throughout this section, $\mathcal{C}$ is a pivotal finite tensor category with pivotal structure $\mathfrak{p}_X : X \to X^{**}$ ($X \in \mathcal{C}$).
The aim of this section is to give a characterization of left $\mathcal{C}$-module category of the form $\mathcal{C}_A$ for some symmetric Frobenius algebra $A$ in $\mathcal{C}$. We first recall the definition:

\begin{definition}
  Let $A$ be a Frobenius algebra.
  We say that $A$ is {\em symmetric} \cite{MR2500035} if $\nu_A \circ \mathfrak{p}_A = \id_A$, where $\nu_A : A^{**} \to A$ is the Nakayama isomorphism \eqref{eq:Nakayama-iso-def}. 
\end{definition}

For $V, W \in \mathcal{C}$, there is an isomorphism
\begin{equation*}
  D^{(\mathfrak{p})}_{V,W} : \Hom_{\mathcal{C}}(V \otimes W, \unitobj) \to \Hom_{\mathcal{C}}(W \otimes V, \unitobj), \quad
  b \mapsto D_{V,W}(b) \circ (\mathfrak{p}_W \otimes \id_V),
\end{equation*}
where $D_{V,W}$ is the isomorphism \eqref{eq:Nakayama-iso-D}.
By \eqref{eq:Nakayama-iso-def-2}, a Frobenius algebra $A$ is symmetric if and only if $D^{(\mathfrak{p})}_{A,A}(\beta) = \beta$, where $\beta : A \otimes A \to \unitobj$ is the pairing on $A$ induced by the Frobenius form of $A$.

Let $\mathcal{M}$ be a finite left $\mathcal{C}$-module category with relative Serre functor $\Ser$. We can make $\Ser$ a `non-twisted' left $\mathcal{C}$-module endofunctor on $\mathcal{M}$ by
\begin{equation*}
  X \catactl \Ser(M)
  \xrightarrow{\quad \mathfrak{p}_X \catactl \id_{\Ser(M)} \quad}
  X^{**} \catactl \Ser(M)
  \xrightarrow{\quad \mathfrak{s}_{X,M} \quad} \Ser(X \catactl M)
  \quad (X \in \mathcal{C}, M \in \mathcal{M}), 
\end{equation*}
where $\mathfrak{s}$ is the twisted left $\mathcal{C}$-module structure of $\Ser$.
By extending the definition given in \cite{MR3435098,2019arXiv190400376S}, we introduce the following terminology:

\begin{definition}
  A {\em pivotal structure} of a finite left $\mathcal{C}$-module category $\mathcal{M}$ is an isomorphism $\id_{\mathcal{M}} \to \Ser$ of left $\mathcal{C}$-module functors, where $\Ser$ is the relative Serre functor of $\mathcal{M}$ regarded as a left $\mathcal{C}$-module endofunctor on $\mathcal{M}$ as above.
  Equivalently, it is an isomorphism $\tilde{\mathfrak{p}} : \id_{\mathcal{M}} \to \Ser$ of functors satisfying the equation
  \begin{equation}
    \label{eq:def-pivotal}
    \tilde{\mathfrak{p}}_{X \catactl M}
    = \mathfrak{s}_{X,M} \circ (\mathfrak{p}_X \catactl \tilde{\mathfrak{p}}_{M})
    \quad (X \in \mathcal{C}, M \in \mathcal{M}).
  \end{equation}
\end{definition}

Let $M \in \mathcal{M}$ be a $\mathcal{C}$-projective object.
We have proved in Lemma \ref{lem:Fb-alg-rel-Serre} that an isomorphism $p: M \to \Ser(M)$ in $\mathcal{M}$ induces a Frobenius form on $\iEnd(M)$. If the isomorphism $p$ is a component of a pivotal structure of $\mathcal{M}$, then $\iEnd(M)$ is symmetric. Namely,

\begin{lemma}
  \label{lem:symm-Frob-construction}
  Let $\mathcal{M}$ be a pivotal finite left $\mathcal{C}$-module category with structured relative Serre functor $(\Ser, \mathfrak{s}, \phi)$ and pivotal structure $\tilde{\mathfrak{p}}$. For every non-zero $\mathcal{C}$-projective object $M \in \mathcal{M}$, the algebra $A := \iEnd(M)$ is a symmetric Frobenius algebra in $\mathcal{C}$ with the Frobenius form given by
  \begin{equation}
    \label{eq:def-Frobenius-trace-piv}
    \lambda_M := \itrace_M \circ \iHom(M, \tilde{\mathfrak{p}}_M):
    A \to \unitobj,
  \end{equation}
  where $\itrace$ is the trace associated to $(\Ser, \mathfrak{s}, \phi)$.
\end{lemma}
\begin{proof}
  This lemma has been proved in \cite{2019arXiv190400376S} under the assumption that $\mathcal{M}$ is exact. The same proof applies for the non-exact case. For reader's convenience, we present an improved proof. For $X \in \mathcal{C}$, we have
  \begin{align*}
    \lambda_{X \catactl M}
    & = \itrace_{X \catactl M} \circ \iHom(X \catactl M, \tilde{\mathfrak{p}}_{X \catactl M}) \\
    & = \itrace_{X \catactl M} \circ \iHom(X \catactl M, \mathfrak{s}_{X, M})
      \circ \iHom(M, \mathfrak{p}_X \catactl \tilde{\mathfrak{p}}_M) \\
    & = \eval_{X^*} \circ (\id_{X^{**}} \otimes \itrace_M \otimes \id_{X^*})
      \circ \mathfrak{c}_{X^*,M,M,X}^{-1}
      \circ \iHom(M, \mathfrak{p}_X \catactl \tilde{\mathfrak{p}}_M) \\
    & = \eval_{X^*} \circ (\mathfrak{p}_X \otimes \lambda_M \otimes \id_{X^*})
      \circ \mathfrak{c}_{X^*,M,M,X}^{-1},
  \end{align*}
  where the first and the last equality follows from~\eqref{eq:def-Frobenius-trace-piv}, the second from \eqref{eq:def-pivotal} and the third from \eqref{eq:itrace-2}. Hence,
  \begin{equation}
    \label{eq:lem-symm-Frob-construction-proof-1}
    \eval_{A^*} \circ (\mathfrak{p}_A \otimes \lambda_M \otimes \id_{A^*})
    = \lambda_{A \catactl M} \circ \mathfrak{c}_{A^*,M,M,A}.
  \end{equation}
  Now we compute $D^{(\mathfrak{p})}_{A,A}(\lambda_M \circ m)$ as follows:
  \begin{align*}
    & D^{(\mathfrak{p})}_{A,A}(\lambda_M \circ m)
    = \eval_{A^*} \circ (\id_{A^{**}} \otimes \lambda_M \otimes \id_{A^*}) \circ (\mathfrak{p}_A \otimes \icomp_{M,M,M} \otimes \coev_A) \\
    & = \lambda_{A \catactl M} \circ \mathfrak{c}_{A^*,M,M,A} \circ (\id_{A} \otimes \icomp_{M,M,M} \otimes \coev_A) \\
    & = \lambda_{A \catactl M} \circ \icomp_{A \catactl M,M, A \catactl M}
      \circ (\mathfrak{a}_{A,M,M} \otimes \mathfrak{b}_{M,M,A^*})
      \circ (\id_A \otimes \id_A \otimes \coev_A) \\
    & = \lambda_{A \catactl M} \circ \icomp_{A \catactl M,M, A \catactl M}
      \circ (\mathfrak{a}_{A,M,M} \otimes (\iHom(\ieval_{M,M},M) \circ \icoev_{M,\unitobj})) \\
    & = \lambda_{M} \circ \icomp_{M, M, M}
      \circ ((\iHom(\ieval_{M,M}, M) \circ \mathfrak{a}_{A,M,M}) \otimes \icoev_{M,\unitobj})) \\
    & = \lambda_{M} \circ \icomp_{M, M, M}
      \circ (\icomp_{M,M,M} \otimes \icoev_{M,\unitobj})
      = \lambda_{M} \circ m,
  \end{align*}
  where the second equality follows from \eqref{eq:lem-symm-Frob-construction-proof-1}, the third from \eqref{eq:iHom-composition-a-2} and~\eqref{eq:iHom-composition-b-2}, the fourth from \eqref{eq:iHom-composition-b-1}, the fifth from the (di-)naturality of $\icomp$ and $\itrace$, the sixth from \eqref{eq:iHom-composition-a-1}, and the last from that $\icoev_{M,\unitobj}$ is the unit of the algebra $A = \iEnd(M)$. The proof is done.
\end{proof}

\begin{theorem}
  \label{thm:symm-Frob-1}
  For a finite left $\mathcal{C}$-module category $\mathcal{M}$, the following are equivalent:
  \begin{enumerate}
  \item $\mathcal{M}$ admits a pivotal structure.
  \item There exists a symmetric Frobenius algebra $A$ in $\mathcal{C}$ such that $\mathcal{M} \approx \mathcal{C}_A$ as left $\mathcal{C}$-module categories.
  \end{enumerate}
\end{theorem}
\begin{proof}
  Suppose that (1) holds.
  We choose a $\mathcal{C}$-progenerator $G \in \mathcal{M}$ and consider the algebra $A = \iEnd(G)$. Then $\mathcal{M} \approx \mathcal{C}_A$ as left $\mathcal{C}$-module categories. Lemma \ref{lem:symm-Frob-construction} implies that $A$ is symmetric Frobenius. Thus (2) holds.

  To prove the converse, it suffices to show that $\mathcal{C}_A$ admits a pivotal structure if $A$ is a symmetric Frobenius algebra in $\mathcal{C}$. For such an algebra $A$, by Lemma \ref{lem:Fb-alg-rel-Serre}, the relative Serre functor $\Ser$ of $\mathcal{C}_A$ is given by $\Ser(M) = (M^{**})_{(\mathfrak{p}_A)}$ for $M \in \mathcal{C}_A$. Hence $\mathcal{C}_A$ has a pivotal structure induced by that of $\mathcal{C}$. The proof is done.
\end{proof}

It is well-known that the class of symmetric Frobenius algebras in $\Vect$ is closed under Morita equivalence. Theorem \ref{thm:symm-Frob-1} implies that the same holds for symmetric Frobenius algebras in a pivotal finite tensor category:

\begin{theorem}
  \label{thm:symm-Frob-2}
  An algebra in $\mathcal{C}$ that is Morita equivalent to a symmetric Frobenius algebra in $\mathcal{C}$ is also symmetric Frobenius.
\end{theorem}
\begin{proof}
  Let $A$ be a symmetric Frobenius algebra in $\mathcal{C}$, and let $B$ be an algebra in $\mathcal{C}$ that is Morita equivalent to $A$.
  Then there exists a $\mathcal{C}$-progenerator $P \in \mathcal{C}_A$ such that $B \cong \iEnd(P)$ as algebras in $\mathcal{C}$. By Theorem~\ref{thm:symm-Frob-1}, $\mathcal{C}_A$ has a pivotal structure. Thus, by Lemma~\ref{lem:symm-Frob-construction}, $B$ is symmetric Frobenius. The proof is done.
\end{proof}


\def\cprime{$'$}
\newcommand{\germ}{\mathfrak}

\end{document}